\documentclass[12pt]{amsart} 

\usepackage{graphicx,amssymb,colordvi,textcomp,latexsym} 

\usepackage{tikz}
\usepackage{ulem} %to use \sout{}

\usepackage[title]{appendix}
\usepackage[showframe=false]{geometry}
\usepackage{changepage}

\usepackage{adjustbox}
\usetikzlibrary{arrows,shapes, positioning, matrix, patterns, decorations.pathmorphing, arrows.meta,automata}

\usepackage{color}
\usepackage{colortbl}
\usepackage{soul}

\usepackage{array}
\usepackage[hidelinks]{hyperref}
\usepackage{amsfonts}
\usepackage{amsmath}
\usepackage{amsthm}

\usepackage{todonotes}
\usepackage{enumitem}   

\newcommand{\ignore}[1]{}

\newcommand{\R}{\mbox{$\mathbb{R}$}}

\newcommand{\ZZ}{\mathbb{Z}}

\newcommand{\TT}{\mathcal{T}}

\newcommand{\beqn}{\begin{eqnarray*}}
\newcommand{\eeqn}{\end{eqnarray*}}

\newcommand{\id}{\mathrm{Id}}

\newtheorem{lemma}{Lemma}[section]

\newtheorem{prop}[lemma]{Proposition}

\newtheorem{cor}[lemma]{Corollary}
\theoremstyle{definition}
\newtheorem{Def}[lemma]{Definition}
\newtheorem{exam}[lemma]{Example}

\theoremstyle{remark}
\newtheorem{rem}[lemma]{Remark}
\newtheorem{rems}[lemma]{Remarks}

\newcommand{\sot}{\mbox{\footnotesize $\frac{1}{3}$}}

\title[Excitatory-Inhibitory Networks]{Classification of 2-node Excitatory-Inhibitory Networks}
\author{Manuela Aguiar}
\address{Manuela Aguiar, Centro de Matem\'atica da Universidade do Porto (CMUP), Faculdade de Ci\^encias, Universidade do Porto, Rua do Campo Alegre s/n, 4169-007 Porto, Portugal\newline 
Faculdade de Economia, Universidade do Porto, Rua Dr Roberto Frias, 4200-464 Porto, Portugal}
\email{maguiar@fep.up.pt}
\author{Ana Dias}
\address{Ana Dias, Centro de Matem\'atica da Universidade do Porto (CMUP), Departamento de Matem\'atica, Faculdade de Ci\^encias, 
Universidade do Porto, Rua do Campo Alegre s/n, 4169-007 Porto, Portugal}
\email{apdias@fc.up.pt}
\author{Ian Stewart}
\address{Ian Stewart, Mathematics Institute, University of Warwick, Coventry CV4 7AL, United Kingdom}
\email{i.n.stewart@warwick.ac.uk}

\keywords{
excitatory-inhibitory network, excitatory and inhibitory connections, ODE-equivalence
}
\subjclass[2020]{
Primary: 92C42, 37N25, 37C20; Secondary: 92B20
} 
\date{\today}

\begin{document}

\allowdisplaybreaks

\begin{abstract} 
We classify connected 2-node excitatory-inhibitory networks under various conditions. 
We assume that, as well as for connections, there are two distinct node-types, excitatory and  inhibitory. 
In our classification we consider  four different types of excitatory-inhibitory networks: restricted, partially 
restricted, unrestricted and completely unrestricted. For each type we give two different classifications. Using 
results on ODE-equivalence and minimality, we classify the ODE-classes and present a minimal representative 
for each ODE-class. We also classify all the networks with valence $\le 2$. These classifications are up to 
renumbering of nodes and the interchange of `excitatory' and `inhibitory' on nodes and arrows.These 
classifications constitute a first step towards analysing dynamics and bifurcations
of excitatory-inhibitory networks. The results have potential applications to biological network models, 
especially neuronal networks, gene regulatory networks, and synthetic gene networks.
\end{abstract}

\maketitle

\section{Introduction}

In many biological networks there is a key distinction between {\it excitatory} and {\it inhibitory} connections. These terms are common in neuroscience; 
 in genetics  the corresponding connections are usually called {\it activators} and {\it repressors}. For simplicity we adopt the excitatory/inhibitory
terminology in this paper.
 
A connection from node $i$ to node $j$ is excitatory when activation of node 
$i$ makes node $j$ more likely to become active. 
A connection is inhibitory if activation of node $i$ makes node $j$ less likely to become active.
The precise level of activity is determined by the detailed dynamics 
of the nodes and arrows; in particular, by the strength of the connection. 
Examples where this distinction is of central importance are networks of neurons, such as the 
connectome of an organism (Hagmann \cite{H05}, Sporns {\it et al.} \cite{STK05})
 and gene regulatory networks (GRNs), Liu {\it et al.} \cite{LLC19}.

A mathematical framework for understanding collective dynamics on coupled networks, such as synchronization and synchrony-breaking or synchrony-preserving bifurcations,
is supplied by the  theory of {\it (coupled cell) networks} and corresponding  {\it network dynamical systems} or {\it admissible ODEs}, see Golubitsky {\it et al.}~\cite{SGP03, GST05, GS06,GS23} and Field~\cite{F04}. 
In this formalism, nodes and connections (`arrows') are assigned `types', 
which constrain the class of admissible ODEs. For classification
purposes we distinguish excitatory nodes from inhibitory
ones by assigning them different types, but make no other assumptions.
See Section \ref{S:EandI}. The aim of this paper is to classify 2-node networks
 with two different types of connection, under a variety of extra conditions,
summarized in Table \ref{T:classifications}.  A companion paper \cite{ADS3node} builds on
these results to extend the classifications to 3-node excitatory-inhibitory networks.

In neuronal networks, as a general (though not universal) rule, the type of connection is determined by its tail node. 
In other words, nodes (as well as connections) are of two types, {\it excitatory} and {\it inhibitory}. Excitatory nodes output excitatory signals and inhibitory nodes output inhibitory signals. 
This condition does not apply to GRNs: a single node can have both excitatory and inhibitory outputs. 

\subsection{Motivation from Previous Work}

Mathematical models of GRNs and results on robust synchronization, based on the existing theoretical results in the above network formalism, have been obtained by Aguiar, Dias and Ruan~\cite{ADH22}. A related theory of homeostasis in network dynamics
has been developed by Golubitsky {\it et al.} \cite{GS16b, GS18, GSAHW20, GW19} and applied to GRNs by Antoneli {\it et al.} \cite{AGS18a}. Synchrony-breaking bifurcations for
six small, basic genetic circuits are studied in Makse {\it et al.} \cite{MGMRS23}. 

Even when the full network is large and complex, many subnetworks have been
identified that enable specific, useful functions within that network. 
These subnetworks, often called {\it motifs},
are small subnetworks with a topology that occurs with significantly higher 
frequency than in randomized networks, Milo~{\it et al.}~\cite{MSIKCA02}.
Tyson and Nov\'ak~\cite{TN10} have classified 2- and 3-node motifs in biochemical reaction networks. They analyze the dynamics of these motifs and provide evidence that they can carry out specific functions. Singhania and Tyson~\cite{ST23} point out that 
there are many examples of 
{\it near-perfect adaptive responses}  in the physiology of living cells,  which corresponds to the  transiently dynamics response of a system  to a change in an environmental signal and then returning near perfectly to its pre-signal state, even in the continued presence of the signal. Using an evolutionary search procedure, \cite{ST23}  
addresses the underlying molecular bases of such behavior. More precisely, 
 it examines a wide class of molecular interaction  $3$-node motif networks for their potential 
  to exhibit near-perfect adaptation.

Leifer {\it et al.}~\cite{LMRASM20} and Morone {\it et al.}~\cite{MLM20} argue that 
in order to be a functional biological building block, a small subnetwork
should not just occur unusually often, but it should offer computational repertoires
analogous to electronic circuits. This idea has
been explored in synthetic biology, showing that small engineered
GRNs can perform logical computations, Dalchau {\it et al.}~\cite{DSHBCP18}.
An example is the toggle switch, which can be made to switch between 
two coexisting stable states by providing suitable inputs~\cite{ASMN03,GCC00,KF05,KVBAWF04}.

Another example occurs in Elowitz and Leibler~\cite{EL00},
who assembled a synthetic oscillatory network
in {\it Escherichia coli}
from three transcriptional repressor systems, naming it the {\it repressilator}.
They observed periodic oscillations with a period of several hours. 
For simplicity, they first analyze an idealized model in which
 all three repressors have the same
dynamics, so the equations are symmetric under the cyclic group $\mathbb{Z}_3$
of order $3$.
Simulations of oscillatory motion produced a discrete rotating wave
with successive $\sot$-period phase shifts. Using methods of network dynamics,
this state is typical of $\mathbb{Z}_3$-symmetric systems, and often occurs
via symmetry-breaking Hopf bifurcation, Golubitsky and Stewart \cite{GS23}.
A stochastic version in~\cite{EL00} produced oscillations\index{oscillation} of different 
and variable amplitudes, but with approximately the same $\sot$-period phase shifts. 

Several other standard synthetic genetic oscillators,
 whose structure is that of
a small network, are surveyed in Purcell {\it et al.}~\cite{PSGB10}.
They include:

(i) The {\it Goodwin oscillator}, Goodwin~\cite{Goodwin63}, which comprises
 a single gene that represses itself.

(ii) {\it Amplified negative feedback oscillators}
 have been studied by
several authors, using different mathematical models.
Guantes and Poyatos~\cite{GP06} explain oscillation as 
a saddle-node bifurcation on an invariant
circle (SNIC). 
Conrad {\it et al.}~\cite{CMNF08} obtain oscillations from
a subcritical Hopf bifurcation.
Atkinson {\it et al.}~\cite{ASMN03} also use Hopf bifurcation.

(iii) The {\it Fussenegger oscillator}
 goes back to Tigges {\it et al.}~\cite{TMSF09}.
It comprises two genes, with both
sense and antisense transcription occurring from one of
them. This creates a delay in the feedback loop, enhancing the ability to oscillate

(iv) The {\it Smolen oscillator},
 Smolen {\it et al.}~\cite{SBB98}, comprises
two genes. One promotes its own transcription and that of the other
gene, while the second  represses its
own transcription and that of the first gene.
Oscillations were first demonstrated mathematically in Hasty
{\it et al.}~\cite{HDRC02} using a simple model, and shown to arise from either a
supercritical or subcritical Hopf bifurcation.

(v) In a {\it variable link oscillator}
one gene regulates itself, and also regulates a second
gene through a variable promoter. The second gene causes repression
via a protease acting on the product of the first gene.
An ODE model is studied in Hasty
{\it et al.}~\cite{HIDMC01}. 

(vi) The {\it metabolator},
 Fung {\it et al.}~\cite{FWSBLL05}, is the first biological oscillator
reported in the literature using metabolites as a core component. 

\subsection{Motifs in {\it Escherichia coli} GRN}

As further motivation,
Figure \ref{F:EI_E_coli} shows eight 3-node motifs from the
gene regulatory network of  {\it Escherichia\ coli}, an organism whose
 genetic regulatory network, compiled by RegulonDB,
  has been characterised in considerable detail
 \cite{GG16}.
The main point of the figure
is to illustrate the presence of nontrivial 3-node motifs in real 
biological networks, but they also illustrate features of the mathematical
classification developed in this paper.
We discuss each motif briefly.

\begin{figure}[h!]
\centerline{
\includegraphics[width=0.9\textwidth]{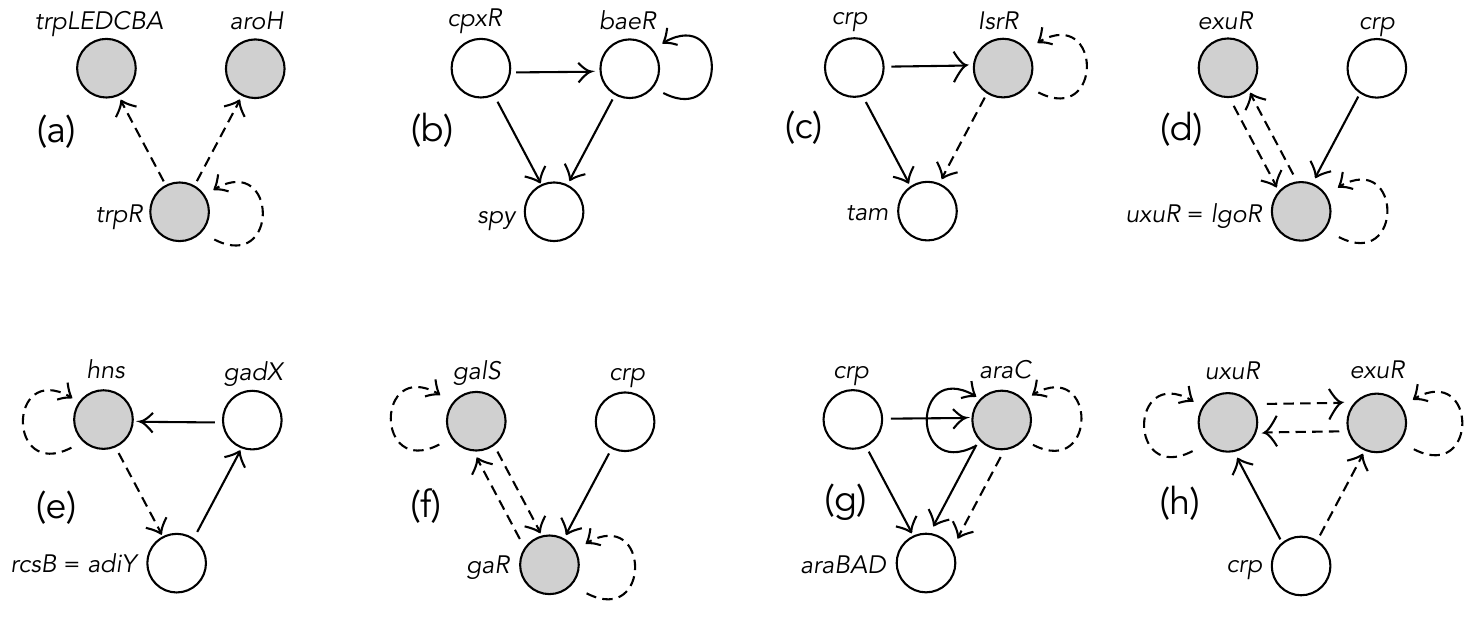}
}
\caption{Eight 3-node motifs realised in {\it E. coli}.}
\label{F:EI_E_coli}
\end{figure}

(a) Autoregulation loop involved in biosynthesis of tryptophan, regulated
by {\it trpR} \cite{GG87}, which represses itself, the gene {\it aroH}, and the 
{\it trpLEDCBA} operon, 
which codes for the enzymes of the 
tryptophan biosynthesis pathway. From \cite{MSDS23}.

(b) Example of a SAT-Feed-Forward-Fiber network.
From \cite{LMRASM20} Fig.1 E.

(c) Example of an UNSAT-Feed-Forward-Fiber network.
From \cite{LMRASM20} Fig.2 F.

(d) Example of a 2-FF network showing quotient by synchrony
of genes {\it uxuR} and {\it IgoR} in a 4-node network in {\it E.coli.}
From \cite{MLM20} Fig. 3B.

(e) Example of a 3-FF network showing quotient by synchrony
of genes {\it rcsB} and {\it adiY} in a 4-node network in {\it E.coli.}
From \cite{MLM20} Fig. 3B.

(f)  Example of a a network where a node feeds 
forward into one node of a toggle-switch. From \cite{M20}.

(g) In the sugar utilisation transcriptional system [24], the arabinose metabolism [25] 
involves the regulation of the {\it araBAD} operon (composed of
genes {\it araB}, {\it araA}, and {\it araD}) by two transcription factors
{\it AraC} and {\it CRP} expressed by genes {\it araC} and {\it crp}, respectively.
From \cite{MMpre}.

(h) Example of a network where a node feeds 
forward into both nodes of a toggle-switch. From \cite{MMpre}. 
\vspace{5pt}

Networks (a) and (b) 
 have arrows (and nodes) of a single type.
Networks (c),(d),(e), and (f) are
what we call REI networks in Definition \ref{def:EIN} below; 
that is, a given node outputs only one type of arrow. 
Networks (g) and (h) are UEI networks: some node outputs arrows of both types.

\subsection{Classification}
Given this degree of interest in small GRNs and their dynamics and bifucations, it is
 of interest to formalize the structure of excitatory-inhibitory (EI) networks
 and to investigate small examples systematically. 
We aim to classify small EI networks under certain conditions. In this paper we classify 2-node EI networks. 
These results are extended to the 3-node case in \cite{ADS3node}.
These classifications can be viewed as a preparatory step towards a
systematic analysis of dynamics and bifurcations in such networks.

We work in the network formalism of~\cite{SGP03, GST05, GS23}, in which nodes (previously called cells) and arrows (connections, directed edges) are partitioned into one or more types. The 
dynamics of the network respects both its topology and the distinction between different types of node or arrow. Such systems of ordinary differential 
equations (ODEs) are said to be {\it admissible} for the network.

The classification of regular
3-node networks, that is,  with only one node and arrow type and where
every node receives one or two arrows,
can be found in \cite{LG06}, along with the classification of all codimension one steady-state and Hopf bifurcations from a synchronous equilibria in these networks.

Our classifications are modelled on that analysis. However, \cite{LG06}
also considers the dynamics
and bifurcations, topics for future work on EI networks. 
Similar classifications have been carried out for regular $4$-node networks by Kamei~\cite{K09},  
$3$-node fully inhomogeneous networks \cite[Section 4.4]{GS23}
 and networks with asymmetric inputs by 
Aguiar, Dias and Soares~\cite{ADS21}. 
Nevertheless, the special structure of excitatory-inhibitory networks
has not previously been addressed in a systematic manner. 

\subsection{Remarks on Excitation and Inhibition}
\label{S:EandI}

The general formalism of \cite{SGP03,GST05,GS23} does not assign a
meaning to the terms `excitatory' or `inhibitory' when applied
to a specific arrow in a network diagram. It does, however, distinguish different
{\it types} of connections (arrows), which is sufficient for classification purposes.
One reason for this approach is that the formalism is designed to apply to 
all ODE models with network structure,
in many of which notions of excitation and inhibition are not relevant. Another is that
couplings are assumed to be determined by general nonlinear functions, 
whereas excitation/inhibition is most natural for linear couplings, see below.

However, the correspondence with biological notions of excitation and inhibition
appears in the general formalism when we introduce an important feature: the
{\it dynamics} of an admissible ODE. It is then
possible to define excitation/inhibition relative to a specific {\it solution} of the ODE,
such as an equilibrium or periodic orbit. This definition usually agrees with 
the terms as used in standard biological models.

In general, an admissible ODE assigns a variable $x_i$ to each node $i$,
representing its dynamic state, and takes the form
\[
\dot x_i = f_i(x_i;x_{i_1}, \ldots, x_{i_m}),
\]
where $i$ indexes the nodes and $i_1, \ldots, i_m$ run through the tail nodes
of all input arrows to node $i$. Moreover, if nodes $i,j$ have the  same node-type and the
same number of input arrows of each arrow-type, we require $f_i=f_j$.

\begin{Def}
\label{D:E/I_def}
Assume for simplicity that $x^*$ is an equilibrium.
Then the $k$th input arrow to node $i$ is:
\begin{equation}
\label{E:ex/inhib_def}
\begin{array}{rcl}
\mbox{{\it excitatory} at}\ x^* & \mbox{if} & \partial_k f_i |_{x^*} > 0, \\
\mbox{{\it inhibitory} at}\ x^* & \mbox{if} & \partial_k f_i |_{x^*} < 0.
\end{array}
\end{equation}
Here $\partial_k f_i$ is the partial derivative of $f_i$ with respect to
the $k$th variable, excluding the first variable $x_i$ which represents
the state of the node concerned. (We do not write these in the more familiar form
$\frac{\partial f_i}{\partial x_{i_k}}$ because the same variable $x_{i_k}$
may appear for different values of $k$, in which case that notation is ambiguous.)
\hfill $\Diamond$
\end{Def}

In many common models, couplings are linear, so that 
\begin{equation}
\label{E:lin_coupling}
\dot x_i = f_i(x_i)+ \alpha_1 x_{i_1}+ \cdots +  \alpha_m x_{i_m}.
\end{equation}
with only $f_i$ nonlinear. Models of this kind go back at least to
Kuramoto \cite{K75,K84}.
Here the coefficients $\alpha_k$ are often called {\it connection strengths} or {\it weights}.
Now arrow $k$ is excitatory (for any solution) if $\alpha_k>0$, and
inhibitory if $\alpha_k<0$. Therefore \eqref{E:ex/inhib_def} is in agreement with standard terminology for such models.

These considerations do not affect the classifications in this paper, or the lists of
admissible ODEs. However, they can be vital when considering dynamics and
bifurcations.

\subsection{Summary of Paper and Main Results}

In this paper we define four different types of  EI networks:
REI, PEI, UEI, and CEI, see Definition~\ref{def:EIN}. These definitions
form the basis of this paper and the companion \cite{ADS3node}.
In both papers we classify networks of these types under various assumptions. 
We also derive subsidiary
classifications under the stronger relation of ODE-equivalence, where
two networks are ODE-equivalent if they have the same space
of admissible ODEs. 
To organise and summarise these results,
Table \ref{T:classifications} lists the main classifications obtained in this paper,
with columns for the bounds on the valence, type of network, number of networks in
the classification, plus
references to associated Figures, Theorems, and lists of 
admissible ODEs.
Even in the 2-node case, we find a rich variety of networks. The 3-node case,
which is more complicated, makes use of the 2-node results.

\begin{table}[!htb]
\begin{center}
{\tiny 
\begin{tabular}{|l|c|c|c|c|}
\hline
 network& number of &figure & theorem & admissible \\
 type & networks & & & ODEs\\
\hline
\hline
 REI & $\infty$ & Figure \ref{fig:generalREI} (left) & Proposition \ref{prop_general_val_2REI} & \eqref{eq:general2REI} \\
\hline
REI  (ODE) & 2 & Figure \ref{fig:ODE_general_REIV} & Proposition \ref{prop:2_node_ODE} & Table \ref{table:ODE_general_REIV} \\
 \hline
REI val $\leq 2$ & 15 & Figure \ref{fig:2NCNREIV2} & Proposition \ref{prop:REI2} & Table \ref{table:admissible_2NCNREIV2} \\
  & \tiny{2 ODE-classes}  &  Figure~\ref{fig:ODE_general_REIV}  &  & Table~\ref{table:ODE_general_REIV} \\
  \hline
  \hline
PEI & $\infty$ & Figure~\ref{fig:generalPEI} & Proposition~\ref{prop_general_val_2REI}  & \eqref{eq:general2REI}   \\
\hline
PEI (ODE) & $\infty$ & Figure~\ref{fig:ODE_general_PEI}   & Proposition \ref{prop:PEI_ODE} &Table~\ref{table:ODE_general_PEI}\\
 \hline
PEI val $\leq 2$ & 15 &    Figure \ref{fig:2NCNREIV2} nodes same type & Proposition \ref{prop:PEI2} & Table \ref{table:admissible_2NCNREIV2}  \\
   &  \tiny{9 ODE-classes} & Figure \ref{fig:ODE_2NCNPEIV2} & &Table \ref{table:ODE_2NCNPEIV2} \\
  \hline
  \hline
UEI (ODE) & $\infty$ & Figure \ref{fig:ODE_general_UEIV} & Proposition \ref{prop:UEI_ODE_classes} & Table \ref{table:ODE_general_UEIV} \\
 \hline
UEI (ODE) val $\leq 2$ & 4 & Figure \ref{fig:ODE_2NCNUEIV2} & Proposition~\ref{L:3.8} & Table \ref{table:ODE_2NCNUEIV2} \\
    & & & &\\
  \hline
 UEI val $\leq 2$ & 53 & Figures \ref{fig:2NCNREIV2} and \ref{fig:2NCNUEIV2} & Proposition \ref{prop:UEI2}  & 
 Tables~\ref{table:admissible_2NCNREIV2} and \ref{table:UEI+ODE}  \\
    &    4 ODE-classes  & Figure \ref{fig:ODE_2NCNUEIV2} & &Table \ref{table:ODE_2NCNUEIV2}\\
 \hline
 \hline
CEI & $\infty$ & Figure~\ref{fig:generalCEI} & Proposition~\ref{lem:general_2_node_CEI} & (\ref{eq:general2UEI})  \\
 \hline
CEI (ODE) &  $\infty$  & --- & Proposition \ref{prop:general_CEI2} & --- \\
  \hline
 CEI val $\leq 2$ & 53 &  Figures \ref{fig:2NCNREIV2} and \ref{fig:2NCNUEIV2} nodes same type    & Proposition \ref{prop:CEI2} & 
 Tables \ref{table:admissible_2NCNREIV2}  and \ref{table:UEI+ODE}  \\
    & \tiny{21 ODE-classes} &  Figures \ref{fig:ODE_2NCNPEIV2} and \ref{fig:ODE_2NCNCEIV2} & & Tables~\ref{table:ODE_2NCNPEIV2} and \ref{table:ODE_2NCNCEIV2} \\
  \hline
\end{tabular}
}
\vspace{5mm}
\caption{List of classifications of connected 2-node EI networks and their locations. (ODE): ODE-equivalence classes. val: valence.}
\label{T:classifications}
\end{center}
\end{table}

Section \ref{S:CEIN} discusses classes of excitatory-inhibitory (EI) networks from
the point of view of the general network formalism of \cite{SGP03,GST05,GS23}.
In this formalism the two arrow-types are different, but the $\pm$ nature of
excitation/inhibition is defined relative to a given dynamical state by \eqref{D:E/I_def}. We
distinguish excitatory nodes from inhibitory ones, giving two distinct node-types.
It is also possible to identify these node-types subject to conditions on their
output arrows. Doing so creates extra synchrony patterns as in 
\cite{LMRASM20,MGMRS23,MLM20}, but we do not discuss these patterns here.
Subsection \ref{S:FD} gives formal definitions of four 
types of excitatory-inhibitory networks, `restricted' (REI), `partially restricted' (PEI), `unrestricted' (UEI), and 
 `completely unrestricted' (CEI).
Subsection \ref{S:AODE} defines the class of admissible ODEs associated with an
EI network, using the Smolen oscillator as a simple example. Adjacency matrices
are also discussed. Subsection \ref{S:BCQN} defines balanced colourings (also called
fibration symmetries) and
the associated quotient networks.

Section \ref{S:CC2EIN} classifies 2-node EI networks under various conditions.
Corresponding admissible ODEs are listed. 
Subsections~ \ref{S:C2REIN}-\ref{S:C2REINV2} are dedicated to the classification of connected 2-node REI networks. In Subsection~\ref{S:C2REINODE}, the classification is done up to ODE-equivalence and in Subsection~\ref{S:C2REINV2} for REI networks of valence $\leq 2$. 
Subsections \ref{S:C2PEI}-\ref{S:C2PEINV2} consider the classification of connected  2-node PEI networks. In Subsection \ref{S:C2PEINODE}, the classification is done up to ODE-equivalence and in Subsection \ref{S:C2PEINV2} is for  2-node PEI networks of valence $\leq 2$.
Subsections \ref{S:C2UEI}-\ref{S:C2UEIV2}  address the classification of connected 2-node UEI networks, where 
Subsection \ref{S:C2UEIODE} classifies  up to ODE-equivalence and 
Subsection \ref{S:C2UEIV2} lists the 2-node UEI networks of valence $\leq 2$.
Finally, Subsections \ref{S:C2CEI}-\ref{S:C2CEINV2} classify connected 2-node CEI networks, where the classification in 
Subsection \ref{S:C2CEINODE} is done up to ODE-equivalence and the classification in 
Subsection \ref{S:C2CEINV2}  is of the networks of valence $\leq 2$.

\section{Classes of Excitatory-Inhibitory Networks}
\label{S:CEIN}

We use the network formalism of \cite{SGP03,GST05}, modified as in \cite{GS23}
 to remove the condition that 
arrows of the same type have heads of the same type and tails of the same type. Instead, we separate the roles of node equivalence 
(formerly cell equivalence) and state equivalence (a consequence of the constraints on admissible maps and ODEs). Nodes are state-equivalent if they have the same state space
in a canonical manner. Intuitively, they are node equivalent 
 if they have the same `internal dynamic'. This change 
broadens the range of networks without affecting the main theorems or their proofs. It is also more natural for the networks considered in this paper.

\begin{rems}
\label{r:CEINrems}
(i) Technically, the node-type can be considered as an arrow-type for a distinguished
`internal arrow'. It constrains the admissible ODEs {\it only} when two nodes $c,d$ of the
same node-type are input equivalent. Otherwise, equal node types have no 
dynamic implications.

(ii) This convention differs considerably from that of many models, where each 
node or arrow contributes a specific {\it term} to the model ODE. The reasoning behind
it is explained in \cite[Sections 9.5, 9.8]{GS23}. A key point is that nodes can synchronize robustly
only when they have the same node type {\it and} have isomorphic input sets.
If the input sets of nodes $c,d$ are not isomorphic, assigning them the same node-type
is redundant and does not constrain the admissible ODEs. Node-types can be
redefined to remove these redundancies.

(iii) In many areas of science, the zero state $x_i = 0$ has a special significance. 
For example, in a biochemical or gene regulatory network, the concentration of
a molecule is always $\geq 0$, and a concentration of 0 implies its absence.
In a neuronal network, a voltage of 0 also has a clear physical meaning. 
In many applications it is assumed that the model ODE has the form 
$\dot x = F(x)$ where $F(0)= 0$, implying the existence of  a fully synchronized
state in which $x = (0,0, \ldots, 0)$. This state exists even when nodes are not
input isomorphic.

(iv)
However, in other areas of application there is nothing special about the value 0.
The general formalism therefore assigns no special significance to the zero state,
and a fully synchronized state may not exist unless the network is homogeneous.

(v) Similar remarks apply to weighted networks, and to network models
of the form \eqref{E:lin_coupling}.
\end{rems}

As remarked in Section \ref{S:EandI}, on a formal level the distinction 
between excitatory arrows and
inhibitory ones reduces to having two different arrow-types $A^E,A^I$. We use {\it only} two
arrow-types, which corresponds to a standard simplified modelling assumption:
all excitatory arrows are identical and all inhibitory arrows
are identical. Without this assumption the lists of networks become much larger.

Node-types pose an additional problem. In some areas of biology, notably
neuroscience, a given node cannot output both an excitatory arrow and an
inhibitory one. In effect, there are two distinct node-types $N^E,N^I$. This assumption leads to
the class of {\it restricted} EI-networks (REI). 

However, this restriction is not universal; for example it often fails in GRNs. Keeping two
node-types but removing the restriction on outputs gives the class of
{\it unrestricted} EI-networks (UEI). 

\begin{rem}
\label{R:balance}
Other classes of EI networks  are also of interest, in particular in connection
with a feature of network dynamics that is important in both theory and applications:
synchrony. Two nodes are {\it synchronous} if they have identical time series for
some solution of the model equations. Here we mention synchrony only in passing, but
some discussion is in order because of its importance. A synchronous state
is {\it robust} if it is determined by a subspace that is invariant under
all admissible maps. Every
robustly synchronous state corresponds to a {\it balanced
colouring} of the nodes, see Section~\ref{S:BCQN}. 
This determines a {\it synchrony pattern} whose
{\it synchrony subspace} is flow-invariant for all admissible ODEs.
In consequence, nodes with different
node-types cannot synchronise robustly.

In this paper we assume that there are two distinct node-types:
$N^E$ (excitatory) and $N^I$ (inhibitory). Since robustly synchronous
nodes must be input-isomorphic, an
$N^E$ node cannot synchronise robustly with an $N^I$ node. This constraint reduces
the range of possible synchrony patterns.
From this viewpoint, REI and UEI networks are never homogeneous (all nodes cannot synchronize simultaneously  and robustly).
For example,  the Smolen network of Figure \ref{fig:Smolen} cannot have a 
nontrivial synchrony pattern. (In the trivial pattern, each node has a different colour.
The trivial colouring is obviously balanced.)
 Another
term for the same idea, discovered independently in a different context,
is {\it fibration symmetry}; see Boldi and Vigna \cite{BV02} and DeVille
and Lerman \cite{DL14}. 

In contrast, the networks studied in \cite{LMRASM20,MGMRS23,MLM20} 
can have more synchrony patterns and, in particular, a total synchrony pattern where all nodes synchronize simultaneously and robustly.
This happens because, in effect, the two node-types
$N^E,N^I$ are considered to be the same. This assumption is reasonable in neuronal
networks, where {\it activation} of a neuron either excites or inhibits
another neuron. This suggests that the difference lies in the outputs, 
or in how a node receiving the output responds to that signal, but
not in the internal dynamics of the neuron. With this assumption, the Smolen network
can synchronize robustly.

In the REI case, 
consider the modification where the two node-types are identified but
the restriction that no node outputs arrows of both types
$A^E,A^I$ is retained. We call these {\it partially restricted} EI-networks (PEI).
The same modification can be done in the UEI case, and amounts to giving
all nodes the same type without arrow-type restrictions.
We call these {\it completely unrestricted} EI-networks (CEI).
\hfill $\Diamond$\end{rem}

For PEI and CEI networks all nodes have the same node-type, which implies a change to 
the admissible ODEs by imposing further constraints on
the component functions, and so extra synchrony patterns can arise.

\subsection{Formal Definitions} 
\label{S:FD}

We define four classes of {\it excitatory-inhibitory networks.}

\begin{Def} \label{def:EIN} 

Consider the following four conditions on a network ${\mathcal G}$:

(a) There are two distinct node-types, $N^E$ and $N^I$.

(b) There are two distinct arrow-types, $A^E$ and $A^I$.

(c) If $e \in A^E$ then $\TT (e) \in N^E$.

(d) If $e \in A^I$ then $\TT (e) \in  N^I$, 

\noindent
where $\TT(e)$ indicates the tail node of arrow $e$. 

Then:

${\mathcal G}$ is a {\it restricted excitatory-inhibitory network} ({\it REI network}) if it
satisfies conditions (a), (b), (c) and (d). 

${\mathcal G}$ is an {\it unrestricted excitatory-inhibit\-ory network} ({\it UEI network})
if it satisfies conditions (a) and (b). 

${\mathcal G}$ is a {\it partially restricted excitatory-inhibitory network} ({\it PEI network}) 
if the node-types $N^E$ and $N^I$ are identified (so (a) fails to apply)
and it satisfies condition (b). Conditions (c) and (d) fail to apply, but each node outputs only one type of arrow.

${\mathcal G}$ is a {\it completely unrestricted excitatory-inhibitory network} ({\it CEI network}) 
if the node-types $N^E$ and $N^I$ are identified (so (a) fails to apply) and it satisfies condition (b). 
Conditions (c) and (d) fail to apply and 
nodes can 
output the two types of arrow.
\hfill $\Diamond$
\end{Def}

\begin{rems}
\label{R:2.3.1}

(i) REI networks are, in particular, UEI networks and PEI networks are, in particular, CEI networks.

(ii) A PEI network has only one node-type and each node outputs arrows of only one type,
but there are still two distinct arrow-types.
 
(ii) CEI networks employ the convention of \cite{LMRASM20,MLM20}, in which
there is one node-type but two distinct arrow-types. 
Each node can output arrows of either arrow-type. This modification
seems inappropriate for neuronal networks, where excitatory neurons have
different internal dynamics from inhibitory ones and networks are REI, but 
it can be appropriate for GRNs.
\end{rems}

\vspace{5pt}
\paragraph{\bf Conventions}

The following conventions are used throughout the paper without further mention,
except as an occasional reminder for clarity. 

(a)	We represent type $N^E$ nodes by white circles and type $N^I$ nodes by grey circles. Type $A^E$ arrows are solid and type
$A^I$ arrows are dashed. (Various other conventions for excitatory/inhibitory arrows
are found in the literature; this one is chosen for convenience.) 
 
 (b) All classifications are stated up to {\it renumbering} of nodes and {\it duality}; that is,
 interchange of `excitatory' and `inhibitory' on nodes and arrows: 
 $N^E \leftrightarrow N^I$ and $A^E \leftrightarrow A^I$.
 \hfill $\Diamond$
 
\begin{Def} \label{Def:input_equiv}
(a) Given a node $i$, denote the set of excitatory arrows directed to $i$ by $I^E(i)$ and the set of inhibitory arrows directed to $i$ by $I^I(i)$. We call $I^E(i)$ and $I^I(i)$ the {\it excitatory} and {\it inhibitory input sets} of $i$, respectively. The {\it input set} of $i$ is then $I(i) = I^E(i) \cup I^I(i)$. The {\it valence (degree, in-degree)} of $i$ is the cardinality $\#I(i)$ of $I(i)$. 

(b) Two nodes $i$ and $j$ are {\it input equivalent} if they have the same node-type and
 there is an arrow-type preserving bijection between the corresponding input sets $I(i)$ and $I(j)$; that is, when $\#I^E(i) = \#I^E(j)$ and $\#I^I(i) = \#I^I(j)$. We write $i \sim_I j$.
 The relation $\sim_I$ is an equivalence relation, which partitions the set of nodes into
 disjoint {\it input classes}. 

(c) A network in which all nodes are input equivalent is {\it homogeneous}. Otherwise, it is
 {\it inhomogeneous}. 
\hfill $\Diamond$
\end{Def}

\begin{rems}
\label{R:2.3.2}

(i) The definition of synchrony in \cite{GS23,GST05,SGP03} implies that
synchronous nodes must be input equivalent. Thus for EI networks, 
nodes of type $N^E$ cannot synchronize with nodes of type $N^I$ 
as, by Definition~\ref{Def:input_equiv} (b), 
nodes of different types are not input equivalent. 
\hfill $\Diamond$

(ii) Every REI and UEI
network has two distinct node-types,  $N^E $ and $N^I$. 
Thus 
REI and UEI networks are inhomogeneous.

(iii) Since PEI and CEI networks have nodes of the same type, 
there can be homogeneous PEI and CEI networks.
\end{rems}

A network is {\it transitive} if there is a closed  
arrow-path containing every node. 
A non-transitive network is often said to be 
{\it feedforward}. This term is also used for a stronger property:
the nodes can be given a partial ordering such that the tail node of any arrow
is smaller in the ordering than its head node.

\subsection{Admissible ODEs}
\label{S:AODE}
The dynamic evolution of node variables $x_i$ is governed by a 
system of ordinary differential equations, said to be admissible. 
The form of admissible ODEs for EI networks can be deduced from
the general definition in \cite{GS23,GST05,SGP03}, but for convenience we
describe it explicitly. 
We assume in this paper that all nodes have the same state space, 
say $P = \R^m$ for some $m > 0$. By definition, input equivalent nodes 
must have the same node-type. Moreover, for each node $i$ in a given input class, the dynamics is 
governed by the same smooth function, say $f : P^{k+1} \to P$, evaluated at the node 
$i$ and at the node tails of the $n_e$ excitatory and the $n_i$ inhibitory 
arrows targeting that node.  For the special case of EI networks, we have: 

\begin{Def}
A system of ODEs is {\it admissible} for an EI network if it has the form
\[
 \dot{x}_i = f_i\left(x^s_i; \overline{x^+_{i_1}, \ldots, x^+_{i_{n_e}}}; \overline{x^-_{i_{n_e+1}}, \ldots, x^-_{i_{n_e+n_i}}}\right) 
\]
where $x^s_i \in \{x_i, x^+_i,x^-_i \}$ and 
the overlines indicate that the function $f_i$ is symmetric in the 
overlined variables.
The node variables are indexed by $i$. 
The {\it multiset} of all tail nodes of
input arrows is the union of two subsets:  
the multiset $\{i_1, \ldots, i_{n_e}\}$ of all tail nodes of the
excitatory input set of node $i$, and the multiset $\{i_{n_e+1}, \ldots, i_{n_e+n_i}\}$
of all tail nodes of  the inhibitory input set of node $i$. 
The functional notation converts
these multisets into tuples of the corresponding variables. 
We use the superscripts $+$ and $-$, as a notation convention, to make the distinction between the input variables corresponding to tail nodes in the excitatory and in the inhibitory input sets, respectively. Analogously, when there are two distinct node-types $N^E$ and $N^I$, we use the superscripts $+$ and $-$ to make the distinction between the state variable of excitatory and inhibitory nodes. Otherwise, no subscript is used.

Moreover, if nodes $i,j$ are in the same input class then $f_i=f_j$.
\hfill $\Diamond$
\end{Def}

\begin{rems}
(i) {\it Multiple arrows} are permitted. That is,
distinct arrows in $I^E(i)$ can have the same tail node, which is why we use multisets.
The same goes for $I^I(i)$. 
(ii) {\it Self-loops} are also permitted. That is, a node can input an arrow to itself.
In biology this is called {\it autoregulation}. 
\hfill $\Diamond$
\end{rems}

If nodes $i$ and $j$ are input equivalent then
$\#I^E(i) = \#I^E(j)$ and $\#I^I(i) = \#I^I(j)$.
Therefore, if there is an  arrow in $I(i)$ of a certain arrow-type, 
then there is also an arrow in $I(j)$ of the same type. The evolution in time of node $j$ is defined similarly to that of node $i$, using the same function $f$ evaluated at 
 $x_j$ and at the corresponding node tails. 
 The evolution of nodes in different input classes is governed by 
 different functions $f_i$, one for each input class. 

\begin{exam} \label{Ex:smolen}
The {\it Smolen oscillator} in Figure~\ref{fig:Smolen} is an REI network. 
 Node $1$ is type $N^E$ and node 
$2$ is type $N^I$. There are two  type $A^E$ arrows; one from $1$ to itself and the other from $1$ to $2$. 
There are two type $A^I$ arrows; one from $2$ to itself and the other from $2$ to $1$. 

Each node has excitatory and inhibitory input sets with cardinality $1$.
Both nodes are autoregulatory. 

The node-types are different so this network is inhomogeneous. 
It is also {\it not} $\ZZ_2$ symmetric. Admissible ODEs are:
\begin{equation}
\label{E:smolen}
\begin{array}{l}
\dot{x}_1 = f(x^+_1; x^+_1; x^-_2)\\
\dot{x}_2 = g(x^-_2; x^+_1; x^-_2)
\end{array}\, .
\end{equation}
Here, $x^+_1,x^-_2 \in P$ where $P$ is the node state space and $f,g:\, P^3  \to P$ are smooth functions. This differs from the total phase space, which is $P^2$; such 
a difference occurs when there are multiple arrows or self-loops.
With suitable dynamics modelling a GRN this network exhibits periodic oscillations, Purcell {\it et al.}~\cite{PSGB10}.
\hfill $\Diamond$
\end{exam}

\begin{rem} \label{Rmk:PEIsmolen}
Considering the network in Example~\ref{Ex:smolen}
with the alternative convention on node-types in PEI networks, 
all nodes are input equivalent, so
the functions $f$ and $g$ are equal. Now there is a balanced colouring
(fibration symmetry) in which $x_1 = x_2$ 
and the diagonal subspace $\Delta = \{x:\, x_1=x_2\}$
is flow-invariant for any admissible ODE. Thus synchronous states
can occur robustly.
However, the Smolen oscillator network
remains asymmetric because there are
 two arrow types. 
\hfill $\Diamond$
\end{rem}

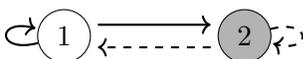
\begin{figure}
\begin{tikzpicture}
 [scale=.15,auto=left, node distance=2cm
 ]
 \node[fill=white,style={circle,draw}] (n1) at (4,0) {\small{1}};
  \node[fill=black!30,style={circle,draw}] (n2) at (20,0) {\small{2}};
  \path
   (7,1)  [->] edge[thick] node {}  (17,1)
 (n1)  [line width=1pt,->]  edge[loop left=90,thick] node {} (n1);
 \path 
 (17,-1)  [dashed,->] edge[thick] node {}  (7,-1)
 (n2)  [dashed,->]  edge[loop right=90,thick] node {} (n2); 
 \end{tikzpicture}  
\caption{Smolen oscillator.} \label{fig:Smolen}
\end{figure}

We can represent an $n$-node network by its {\it adjacency matrix}, which is the $n \times n$
matrix $A = (a_{ij})$ such that 
$a_{ij}$ is the number of arrows from node $j$ to node $i$. Ordinarily, this representation could lose information, 
because it fails to distinguish the different arrow-types. However, conditions (c) and (d) of Definition~\ref{def:EIN} allow us to deduce the arrow-types for REI networks, provided we
know the node-types of nodes $i$ and $j$. 
The networks that are PEI, CEI, and
 UEI require two adjacency matrices, one for each arrow-type; see the next example.

\begin{exam}
\label{ex:smolen_adj}
The adjacency matrix of the Smolen network in Figure~\ref{fig:Smolen}  is 
$$
\left[
\begin{array}{cc}
1 & 1 \\
1 & 1 
\end{array}
\right]\, .
$$
For some purposes, such as ODE-equivalence, we must distinguish node- and arrow-types and equip each with its own adjacency matrix. Here there are four:
$$
\begin{array}{ll} 
\mbox{Node-type $N^E$: } 
\left[
\begin{array}{cc}
1 & 0 \\
0 & 0 
\end{array}
\right]; \quad 
\mbox{Node-type $N^I$: } 
\left[
\begin{array}{cc}
0& 0 \\
0 & 1 
\end{array}
\right]; \\ 
\ \\
\mbox{Arrow-type $A^E$: } 
\left[
\begin{array}{cc}
1 & 0 \\
1 & 0 
\end{array}
\right]; \quad 
\mbox{Arrow-type $A^I$: } 
\left[
\begin{array}{cc}
0& 1 \\
0 & 1 
\end{array}
\right]\, .
\end{array}
$$
Alternatively, if we assume the Smolen network to be a PEI network (just one node-type), then there 
are three adjacency matrices:
$$
\begin{array}{ll} 
\mbox{Node-type $N^E = N^I$: } 
\left[
\begin{array}{cc}
1 & 0 \\
0 & 1 
\end{array}
\right]; \\ 
\ \\
\mbox{Arrow-type $A^E$: } 
\left[
\begin{array}{cc}
1 & 0 \\
1 & 0 
\end{array}
\right]; \quad 
\mbox{Arrow-type $A^I$: } 
\left[
\begin{array}{cc}
0& 1 \\
0 & 1 
\end{array}
\right]\, .
\end{array}
$$
\hfill $\Diamond$
\end{exam}

\begin{rem} \label{rem:ODEequiv}
(i) Different networks sometimes have the same set of admissible ODEs,
 for any choice of node state spaces, when their nodes are identified
 by a suitable bijection that preserves node state spaces. 
 Such networks are {\it ODE-equivalent}, \cite{DS05,GS23, GST05}.
 
 (ii) By Dias and Stewart~\cite[Theorem 7.1, Corollary 7.9]{DS05}, two networks with the same number of nodes are ODE-equivalent if and only if, for a suitable identification of nodes, they have
 the same vector spaces of {\it linear} admissible maps when node state spaces are $\R$.
 Equivalently, the adjacency matrices of all node- and arrow-types span the same space.
  Trivially, this remains true when restricting to the set of EI networks. 
  
(iii) Following Aguiar and Dias~\cite{AD07,AD07a},  given an ODE-class of EI networks, we can distinguish a subclass containing the EI networks in the ODE-class that have a minimal number of arrows. This is a {\it minimal subclass}. As an example, it is proved in ~\cite{AD07} that the ODE-class of every homogeneous regular 
$n$-node network with only one arrow-type contains a unique minimal network. In general, the minimal class of an ODE-class need not be a singleton. 

(iv) For REI networks, we saw that the node-types determine the arrow-types
and the adjacency matrices naturally decompose into four blocks.
The linear condition in (ii) preserves this decomposition, so two REI networks are
ODE-equivalent if and only if these components are separately ODE-equivalent. 
Moreover, note that an ODE-class for an REI network always contains UEI networks 
that are not REI networks. Nevertheless, 
the methods of 
\cite{AD07,AD07a} 
can easily be adapted to
prove that there is an ODE-equivalence that preserves the
REI structure 
for minimal subclasses. 
The corresponding issue for UEI networks is not 
true. 
For example, consider the 2-node UEI-network with arrow-type adjacency
matrices $$
\mbox{Arrow-type $A^E$: } 
A_3 = 
\left[
\begin{array}{cc}
0 & 0 \\
1 & 0 
\end{array}
\right]; \quad 
\mbox{Arrow-type $A^I$: } 
A_4 = 
\left[
\begin{array}{cc}
0& 1 \\
1 & 0 
\end{array}
\right]\, .
$$
Now observe that $A_4 - A_3 = 
\left[
\begin{array}{cc}
0& 1 \\
0 & 0 
\end{array}
\right]$.
Thus, the UEI-network is ODE-equivalent to the 
minimal 
REI-network with arrow-types adjacency matrices
$\left[
\begin{array}{cc}
0 & 0 \\
1 & 0 
\end{array}
\right]$ and 
$ \left[
\begin{array}{cc}
0& 1 \\
0 & 0 
\end{array}
\right]$. 
\hfill $\Diamond$
\end{rem} 

 \begin{exam} 
In Sections~\ref{S:C2REINV2} and \ref{S:C2UEIV2}, we determine the ODE-classes for the $2$-node REI and UEI networks, respectively, with valence up to $2$. 
As the ODE-classes for REI networks also contain UEI networks that are not REI networks, we have, as mentioned in Remark~\ref{rmk:UEI_classes}, 
that two of the ODE-classes for UEI networks coincide with the two ODE-classes for REI networks. The other two ODE-classes for UEI networks contain 
only UEI networks that are not REI networks. However, if we consider only the minimal ODE-classes, 
then for the REI case the classes contain only REI networks. For the UEI case, 
two of the classes contain only REI networks and the other two only UEI networks that are not REI networks.
 \hfill $\Diamond$
 \end{exam}

 \begin{exam}
 The  
REI 
network in Figure~\ref{fig:ODE_equiv_Smolen} is ODE-equivalent to the  
REI 
Smolen network in Figure~\ref{fig:Smolen} and it is minimal. 
 Moreover, the admissible ODE \eqref{E:smolen} determines an arbitrary dynamical
system in $(x_1,x_2)$. 
 \hfill $\Diamond$
 \end{exam}
 
 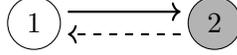
\begin{figure}
\begin{tikzpicture}
 [scale=.15,auto=left, node distance=2cm
 ]
 \node[fill=white,style={circle,draw}] (n1) at (4,0) {\small{1}};
  \node[fill=black!30,style={circle,draw}] (n2) at (20,0) {\small{2}};
  \path
   (7,1)  [->] edge[thick] node {}  (17,1);
  \path 
 (17,-1)  [dashed,->] edge[thick] node {}  (7,-1);
 \end{tikzpicture}  
\caption{The minimal  $2$-node network ODE-equivalent to the Smolen oscillator in Figure~\ref{fig:Smolen}. } \label{fig:ODE_equiv_Smolen}
\end{figure}

\subsection{Balanced Colourings and Quotient Networks}
\label{S:BCQN}

We specialise \cite[Definition~6.4]{SGP03} to EI networks: 

\begin{Def}
Consider an $n$-node EI network. Let $\bowtie$ be an equivalence relation on the set $N^E \cup N^I$ of nodes, refining input equivalence. If $i \bowtie j$ then either $i,j \in N^E$ or $i,j \in N^I$.
 The relation $\bowtie$ is {\it balanced} if, for every two nodes $i,j$ such that $i \bowtie j$, there 
are arrow-type preserving  bijections $\beta^E: I^E(i) \to I^E(j)$ and $\beta^I: I^I(i) \to I^I(j)$ such that if $e \in I^E(i)$ then $\TT(e) \bowtie \TT\left(\beta^E(e)\right)$ and if $e \in I^I(i)$ then $\TT(e) \bowtie \TT\left(\beta^I(e)\right)$.

\hfill $\Diamond$
\end{Def}

\begin{Def} 
Let $\mathcal{G}$ be an EI network in which all node state spaces are $P$,
and let $\bowtie$ be a balanced equivalence relation on the set of nodes. Then  
the {\it polydiagonal (subspace)} $\Delta_{\bowtie} \,\subseteq\, P^n$ 
consists of all $x =(x_1, \ldots, x_n)$ such that $i \bowtie j$ implies
that $x_i = x_j$.
The space $\Delta_{\bowtie}$ is also called a {\it synchrony subspace}. 
 \hfill $\Diamond$
\end{Def}

By \cite[Proposition 10.20 ]{GS23} or \cite[Theorem 6.5]{SGP03}, a colouring $\bowtie$ is balanced  
if and only if $\Delta_{\bowtie}$ is invariant under the flow of every linear network admissible ODE.
This condition holds for any choice of node state spaces.
Equivalently, $\bowtie$ is balanced if and only if $\Delta_{\bowtie}$ is invariant under all  adjacency matrices.

More generally, a vector subspace is a balanced polydiagonal, 
a synchrony subspace, 
 if and only if it is invariant
under any admissible ODE (including nonlinear ones): see \cite[Theorem 10.21 ]{GS23}.
Synchrony subspaces therefore play a crucial role in the dynamics of 
admissible ODEs. See for example \cite{AADF11}.

Following \cite[Section 5]{GST05}, given an EI network $\mathcal{G}$ and a balanced equivalence relation~$\bowtie$ with $k$ classes, we can define
 the {\it quotient network} of $\mathcal{G}$ by $\bowtie$. This is the $k$-node network 
whose nodes correspond to the $\bowtie$-equivalence classes, and whose 
arrows are the projections of those of $\mathcal{G}$, preserving arrow-type. 
It is variously denoted by $\mathcal{G}/{\bowtie}$, $\mathcal{G}^{\bowtie}$, or $\mathcal{G}_{\bowtie}$.

Trivially, we have: 
\begin{prop}
For any EI network $\mathcal{G}$ having a balanced equivalence relation~$\bowtie$ with $k$ classes, the quotient of $\mathcal{G}$ by $\bowtie$ is a $k$-node EI network.
Moreover, the quotient of $\mathcal{G}/{\bowtie}$ is an REI  
(resp. UEI, 
CEI) network if and only if the network $G$ is REI (resp. UEI, CEI).
\hfill $\Diamond$
\end{prop}

\begin{proof}
Let $\mathcal{G}$ be an EI network  having a balanced equivalence relation~$\bowtie$ with $k$ classes. By definition, $\bowtie$ refines the input equivalence relation which  refines that of node equivalence. Thus, the quotient $\mathcal{G}/{\bowtie}$ is a REI  
(resp. UEI) network if and only if the network $G$ is  REI (resp. UEI). Trivially, a quotient of a CEI network is a CEI network.
\end{proof}

\begin{rem} 
In the case of a PEI network, the quotient network 
is always a CEI network, but it might not be a PEI network.
In a PEI network, 
if two nodes of the same colour output arrows of different type (one outputs excitatory arrows and the other inhibitory arrows) then 
in the quotient there is a node that outputs arrows of the two types, 
that is, the quotient is a CEI network. 
As an example, if $\mathcal{G}$ is the Smolen oscillator  in Figure~\ref{fig:Smolen} and we assume that the two nodes have the same type, that is, 
$\mathcal{G}$ is a PEI network, then the equivalence relation $\bowtie\, = \left\{ \{1,2\} \right\}$ is balanced and $\mathcal{G}/{\bowtie}$ is a CEI network. 
In fact, $\mathcal{G}/{\bowtie}$ is the one-node network with two arrows (of different type) to itself and  the synchrony space $\Delta_{\bowtie} = \{x:\, x_1=x_2\}$ 
corresponds to the the diagonal subspace. Recall Example~\ref{Ex:smolen} and Remark~\ref{Rmk:PEIsmolen}.
\hfill $\Diamond$
\end{rem}

\begin{rem} 
By \cite[Theorem 10.28]{GS23} or \cite[Theorem 5.2]{GST05}, every admissible ODE for $\mathcal{G}$ restricted to $\Delta_{\bowtie}$ 
can be identified with an admissible ODE for the quotient $\mathcal{G}/{\bowtie}$. Moreover, every admissible ODE for the quotient $\mathcal{G}/{\bowtie}$ can be identified with the restriction to $\Delta_{\bowtie}$ of an
admissible ODE for $\mathcal{G}$.
\hfill $\Diamond$
\end{rem}

\section{Classification of Connected 2-node Excitatory-Inhibitory Networks}
\label{S:CC2EIN}

With these preliminaries out of the way, we now classify EI networks with two nodes. 
It is enough to classify connected
networks, since components of a disconnected network have fewer nodes
and the $1$-node case is trivial. 

We repeat Convention (b) in Section~\ref{S:FD}:
all classifications are stated up to duality (interchange of the $E$ and $I$ types) and 
renumbering of nodes. We also consider ODE-equivalence. Based on this classification, 
we then list all the connected $2$-node EI networks with valence at most 2 for any node. 
This condition eliminates most multiple arrows, which are uncommon (though possible) in real biological networks.

\subsection{Connected 2-node REI Networks}
\label{S:C2REIN}

We consider first $2$-node REI networks. 
By duality, we assume that node $1$ is excitatory and node $2$ is inhibitory.

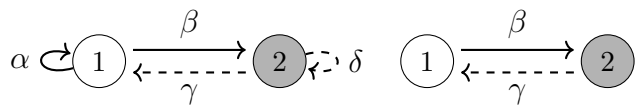
\begin{figure}
\begin{tabular}{ll}
\begin{tikzpicture}
 [scale=.15,auto=left, node distance=2cm
 ]
 \node[fill=white,style={circle,draw}] (n1) at (4,0) {\small{1}};
  \node[fill=black!30,style={circle,draw}] (n2) at (20,0) {\small{2}};
  \path
   (7,1)  [->] edge[thick] node {$\beta$}  (17,1)
 (n1)  [line width=1pt,->]  edge[loop left=90,thick] node {$\alpha$} (n1);
 \path 
 (17,-1)  [dashed,->] edge[thick] node {$\gamma$}  (7,-1)
 (n2)  [dashed,->]  edge[loop right=90,thick] node {$\delta$} (n2); 
 \end{tikzpicture}  & 
 \begin{tikzpicture}
 [scale=.15,auto=left, node distance=2cm
 ]
 \node[fill=white,style={circle,draw}] (n1) at (4,0) {\small{1}};
  \node[fill=black!30,style={circle,draw}] (n2) at (20,0) {\small{2}};
  \path
   (7,1)  [->] edge[thick] node {$\beta$}  (17,1);
 \path 
 (17,-1)  [dashed,->] edge[thick] node {$\gamma$}  (7,-1);
 \end{tikzpicture} 
 \end{tabular}
 \caption{Two $2$-node ODE-equivalent REI networks. Node $1$ is excitatory, node $2$ is inhibitory; the nonnegative integer arrow multiplicities are $\alpha, \beta, \gamma, \delta$. The network is connected when one of $\beta$ or $\gamma$ is nonzero.} \label{fig:generalREI}
\end{figure}

\begin{prop} \label{prop_general_val_2REI}
A $2$-node connected REI network is the network of Figure~{\rm \ref{fig:generalREI}} (left) for some choice of nonnegative integer arrow multiplicities $\alpha, \beta, \gamma, \delta$, where one of $\beta$ or $\gamma$ is nonzero. This network has admissible ODEs
\begin{equation}
\begin{array}{l}
\dot{x}_1 = f(x^+_1; \underbrace{\overline{x^+_1, \ldots, x^+_1}}_{{\alpha}}; \underbrace{\overline{x^-_2, \ldots, x^-_2}}_{\gamma}), \\
\\
\dot{x}_2 =g(x^-_2;\underbrace{\overline{x^+_1, \ldots, x^+_1}}_{\beta}; \underbrace{\overline{x^-_2, \ldots, x^-_2}}_{{\delta}}). \\
\end{array}
\label{eq:general2REI}
\end{equation}
Here $x^+_1, x^-_2 \in \R^k$, so the total state space is $\R^{2k}$, and $f$ and $g$ 
are smooth functions invariant under permutation of the variables under each overline. 

To simplify notation, temporarily omit the $\pm$ superscripts, and define
\begin{equation} 
\begin{array}{l} 
p = (x_1, \underbrace{x_1, \ldots, x_1}_{{\alpha}}, \underbrace{x_2, \ldots, x_2}_{\gamma}); \qquad  
q= (x_2, \underbrace{x_1, \ldots, x_1}_{\beta}, \underbrace{x_2, \ldots, x_2}_{{\delta}});\\
\\
a_1 = \frac{\partial f}{\partial x_1}\left|_{p}\right.; \qquad 
b_1 =  \frac{\partial f}{\partial x_2}\left|_{p}\right.; \qquad 
c_1 = \frac{\partial f}{\partial x_{\alpha + 2}}\left|_{p}\right.; \\
\\
d_1 = \frac{\partial g}{\partial x_1}\left|_{q}\right.; \qquad  
e_1 = \frac{\partial g}{\partial x_2}\left|_{q}\right.; \qquad 
f_1 = \frac{\partial g}{\partial x_{\beta + 2}}\left|_{q}\right.\, .
\end{array}
\label{eq:notation_REI}
\end{equation}
Then the linearization of {\rm (\ref{eq:general2REI})} at  $(x_1,x_2)$ is the $2k \times 2k$ matrix 
\begin{equation}
\left[
\begin{array}{cc}
a_1 + \alpha b_1 & \gamma c_1 \\
\beta e_1 & d_1 + \delta f_1
\end{array}
\right]\, . 
\label{eq:lin_REI}
\end{equation}
\end{prop}

\begin{proof}
Since node $1$ is excitatory the tail of any type $A^E$ arrow is node $1$, and since node $2$ is inhibitory the tail of any type $A^I$ arrow is node $2$. Thus the adjacency matrices are 
\begin{equation}
\label{E:adj_mx_4types}
\begin{array}{ll} 
\mbox{Node-type $N^E$: } 
A_1 = 
\left[
\begin{array}{cc}
1 & 0 \\
0 & 0 
\end{array}
\right]; \quad 
\mbox{Node-type $N^I$: } 
A_2 = 
\left[
\begin{array}{cc}
0& 0 \\
0 & 1 
\end{array}
\right]; \\ 
\ \\
\mbox{Arrow-type $A^E$: } 
A_3 = 
\left[
\begin{array}{cc}
\alpha & 0 \\
\beta & 0 
\end{array}
\right]; \quad 
\mbox{Arrow-type $A^I$: } 
A_4= 
\left[
\begin{array}{cc}
0& \gamma \\
0 & \delta 
\end{array}
\right];
\end{array}
\end{equation}
where at least $\beta$ or $\gamma$ is nonzero to obtain a connected network. 
These conditions correspond to the network in Figure~\ref{fig:generalREI} (left).

Concerning the linearization of equations (\ref{eq:general2REI}) 
at $(x_1,x_2)$ and the notation in
(\ref{eq:notation_REI}), symmetry of $f$ and $g$ implies that 
$$
\begin{array}{l}
b_1 =  \frac{\partial f}{\partial x_2}\left|_{p}\right. =  \frac{\partial f}{\partial x_3}\left|_{p}\right. = \cdots = 
 \frac{\partial f}{\partial x_{\alpha + 1}}\left|_{p}\right., \quad 
c_1 = \frac{\partial f}{\partial x_{\alpha + 2}}\left|_{p}\right. = \frac{\partial f}{\partial x_{\alpha + 3}}\left|_{p}\right.  = \cdots = 
 \frac{\partial f}{\partial x_{\alpha + \gamma + 1}}\left|_{p}\right., \\
 \\
e_1 =  \frac{\partial g}{\partial x_2}\left|_{q}\right. =  \frac{\partial g}{\partial x_3}\left|_{q}\right. = \cdots = 
 \frac{\partial g}{\partial x_{\beta + 1}}\left|_{q}\right., \quad 
f_1 = \frac{\partial g}{\partial x_{\beta + 2}}\left|_{q}\right. = \frac{\partial g}{\partial x_{\beta + 3}}\left|_{q}\right.  = \cdots = 
 \frac{\partial g}{\partial x_{\beta + \delta + 1}}\left|_{q}\right. \, .
\end{array}
$$
From this we obtain (\ref{eq:lin_REI}).
\end{proof}

\begin{rem}
The $2$-node REI network in Figure~\ref{fig:generalREI} (left) for $\alpha= \beta= \gamma= \delta=1$ is the Smolen network in Figure~\ref{fig:Smolen}. 
\hfill $\Diamond$ 
\end{rem}

\subsection{Connected 2-node REI Networks: ODE-classes}
\label{S:C2REINODE}

We now deduce the classification of connected 2-node REI networks up to ODE-equivalence.

\begin{prop} \label{prop:2_node_ODE}
There are exactly two ODE-classes of connected $2$-node REI networks, with representatives {\rm NH1} and {\rm NH2}
pictured in Figure~{\rm \ref{fig:ODE_general_REIV}}. The associated admissible ODEs are stated in Table~{\rm \ref{table:ODE_general_REIV}}. 
\end{prop}

\begin{figure}[!ht]
\begin{center}
{\tiny 
\begin{tabular}{ll} 
NH1 \begin{tikzpicture}
 [scale=.15,auto=left, node distance=1.5cm, 
 ]
 \node[fill=white,style={circle,draw}] (n1) at (4,0) {\small{1}};
  \node[fill=black!30,style={circle,draw}] (n2) at (14,0) {\small{2}};
 \draw[->, thick] (n1) edge[thick] node {}  (n2); 
 \end{tikzpicture}  
\quad & \quad 
NH2 \begin{tikzpicture}
 [scale=.15,auto=left, node distance=1.5cm, 
 ]
 \node[fill=white,style={circle,draw}] (n1) at (4,0) {\small{1}};
  \node[fill=black!30,style={circle,draw}] (n2) at (14,0) {\small{2}};
\path
        (n1) [->,thick]  edge[bend right=10] node { } (n2);
\path 
        (n2) [->,thick,dashed]  edge[bend left=-10] node { } (n1);
 \end{tikzpicture}   
\end{tabular}}
\end{center}
\caption{The two ODE-classes of connected $2$-node REI networks. 
Table~\ref{table:ODE_general_REIV} states the corresponding admissible ODEs. The network NH2 is the minimal network ODE-equivalent to the Smolen network of Figure~\ref{fig:Smolen}.}
\label{fig:ODE_general_REIV}
\end{figure}
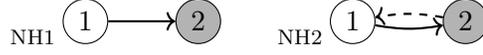

\begin{table}[!htb]
{\tiny 
\begin{tabular}{|l|l|}
\hline 
 &   \\
NH1
$ 
\begin{array}{l}
\dot{x}_1 = f(x^+_1) \\
\dot{x}_2 = g(x^-_2; x^+_1)
\end{array}$
& 
NH2
$ 
\begin{array}{l}
\dot{x}_1 = f(x^+_1; x^-_2) \\
\dot{x}_2 = g(x^-_2; x^+_1)
\end{array}$ \\
 &   \\
\hline  
\end{tabular}
}
\vspace{.2cm}
\caption{Admissible ODEs for the networks in Figure~\ref{fig:ODE_general_REIV}.}
\label{table:ODE_general_REIV}
\end{table}

\begin{proof} 
We recall that our classifications are stated  up to duality and numbering of the nodes. 
Any connected $2$-node REI network is of the form presented in Figure~\ref{fig:generalREI} (left), where $\alpha, \beta, \gamma, \delta$ are nonnegative integers representing arrow multiplicities. The adjacency matrices are stated in \eqref{E:adj_mx_4types},
where at least $\beta$ or $\gamma$ is nonzero to guarantee connectedness.  Trivially, 
$$
\langle  A_1, A_2, A_3, A_4\rangle  = \left\langle  A_1, A_2, 
\left[
\begin{array}{cc}
0 & 0 \\
\beta & 0 
\end{array}
\right], 
\left[
\begin{array}{cc}
0& \gamma \\
0 & 0
\end{array}
\right] \right\rangle \, ,
$$
where angle brackets denote the real subspace spanned by their contents.
Therefore the network in Figure~\ref{fig:generalREI} (left) is ODE-equivalent to the network in Figure~\ref{fig:generalREI} (right). 
If $\gamma=0$, so $\beta\not=0$, then
$$
\langle  A_1, A_2, A_3, A_4\rangle  = \left\langle\langle  A_1, A_2, 
\left[
\begin{array}{cc}
0 & 0 \\
1 & 0 
\end{array}
\right]\right\rangle 
$$
giving network NH1 in Figure~\ref{fig:ODE_general_REIV}. (The case $\beta=0$ and $\gamma\not=0$ is dual.) If both $\beta$ and $\gamma$ are nonzero then 
$$
\langle  A_1, A_2, A_3, A_4\rangle  = \left\langle  A_1, A_2, 
\left[
\begin{array}{cc}
0 & 0 \\
1 & 0 
\end{array}
\right], 
\left[
\begin{array}{cc}
0 & 1 \\
0 & 0 
\end{array}
\right]\right\rangle 
$$
giving network NH2 in Figure~\ref{fig:ODE_general_REIV}. 
\end{proof}

\subsection{Connected 2-node REI Networks with valence up to 2}
\label{S:C2REINV2}

Recall that the valence of a node is the number of input arrows to that node.
Using the classification in the previous section, we now list all connected $2$-node EI networks where all nodes have valence $\leq 2$. 
In what follows, a bound on the valence such as $\leq 2$ means
that the valence of {\it each} node satisfies that bound.

\begin{prop}
There are $15$ connected $2$-node REI networks with valence $\leq 2$. Of these, 
$9$ are in the ODE-class of {\rm NH1} in Figure~{\rm \ref{fig:ODE_general_REIV}}, and $6$ are in the ODE-class of {\rm NH2} in Figure~{\rm \ref{fig:ODE_general_REIV}}.
See Figure~{\rm \ref{fig:2NCNREIV2}} and Table~{\rm \ref{table:admissible_2NCNREIV2}}.  
\label{prop:REI2}
\end{prop}

\begin{proof}
By Proposition~\ref{prop:2_node_ODE} a connected $2$-node REI network with valence $\leq 2$ is ODE-equivalent either to NH1 or to NH2. Each of NH1 and NH2 is a minimal representative of its ODE-class. Network NH1 has only one arrow, which is excitatory and sent by excitatory node 1 to node 2. In an ODE-equivalent network, node 1 can send two excitatory arrows to node 2.
Also, up to ODE-equivalence, node 1 can have no autoregulation or one or two autoregulation excitatory arrows, and node 2 can have no autoregulation or one autoregulation inhibitory arrow.  Considering all the combinations, with valence up to 2, we find 9 networks that are ODE-equivalent to NH1, namely networks (a)-(i) in Figure~\ref{fig:2NCNREIV2}. NH2 has one excitatory arrow, which is sent by the excitatory node 1 to node 2, and one inhibitory arrow, which is sent by the inhibitory node 2 to node 1. 
Up to ODE-equivalence,  node 1 can receive one more inhibitory arrow from node 2 or one autoregulation excitatory arrow.
 Also, up to ODE-equivalence, node 2 can receive one more excitatory arrow from node 1 or one autoregulation inhibitory arrow.
 Considering all the combinations,  up to duality, we find 6 networks that are ODE-equivalent to NH2, namely networks (j)-(o) in Figure~\ref{fig:2NCNREIV2}.  Table~\ref{table:admissible_2NCNREIV2} lists their admissible ODEs. 
\end{proof}

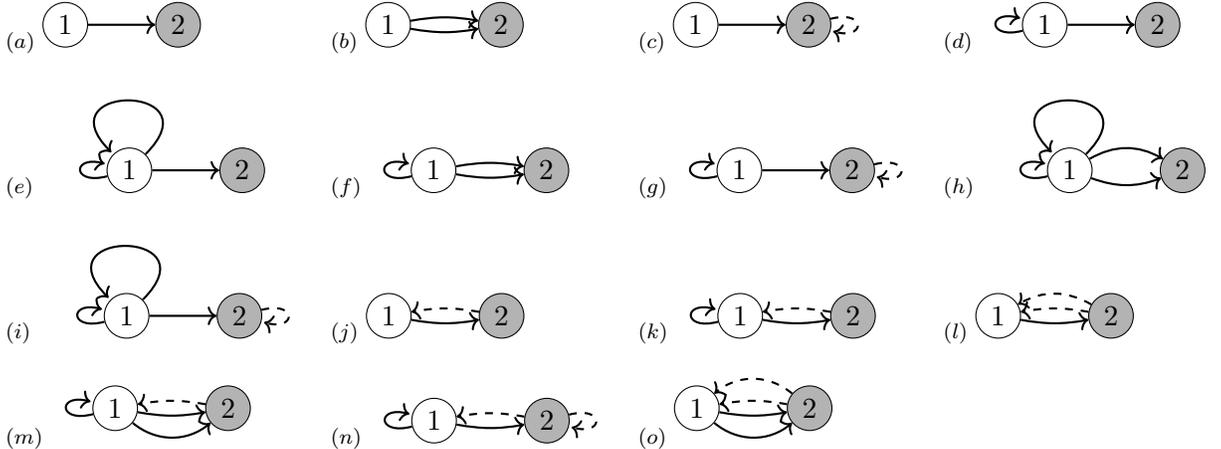
\begin{figure}[!h]
{\tiny 
\begin{center}
\begin{tabular}{llll} 
$(a)$ \begin{tikzpicture}
 [scale=.15,auto=left, node distance=1.5cm, 
 ]
 \node[fill=white,style={circle,draw}] (n1) at (4,0) {\small{1}};
  \node[fill=black!30,style={circle,draw}] (n2) at (14,0) {\small{2}};
 \draw[->, thick] (n1) edge[thick] node {}  (n2); 
 \end{tikzpicture}  
 &  
$(b)$ \begin{tikzpicture}
 [scale=.15,auto=left, node distance=1.5cm, 
 ]
 \node[fill=white,style={circle,draw}] (n1) at (4,0) {\small{1}};
  \node[fill=black!30,style={circle,draw}] (n2) at (14,0) {\small{2}};
 \draw[->, thick] (n1) edge[bend right=10, thick] node {}  (n2); 
 \draw[->, thick] (n1) edge[bend right=-10,thick] node {}  (n2); 
 \end{tikzpicture}  
  &  
$(c)$ \begin{tikzpicture}
 [scale=.15,auto=left, node distance=1.5cm, 
 ]
 \node[fill=white,style={circle,draw}] (n1) at (4,0) {\small{1}};
  \node[fill=black!30,style={circle,draw}] (n2) at (14,0) {\small{2}};
\path
        (n1) [->,solid, thick]  edge[thick] node { } (n2);
\path        
        (n2) [->,dashed, thick]  edge[loop right=90,thick] node { } (n2);
 \end{tikzpicture} 
  &  
$(d)$ \begin{tikzpicture}
 [scale=.15,auto=left, node distance=1.5cm, 
 ]
 \node[fill=white,style={circle,draw}] (n1) at (4,0) {\small{1}};
  \node[fill=black!30,style={circle,draw}] (n2) at (14,0) {\small{2}};
\path
        (n1) [->,thick]  edge[thick] node { } (n2)
        (n1) [->,thick]  edge[loop left=90,thick] node { } (n1);
 \end{tikzpicture} \\
 \\   
 $(e)$ \begin{tikzpicture}
 [scale=.15,auto=left, node distance=1.5cm, 
 ]
 \node[fill=white,style={circle,draw}] (n1) at (4,0) {\small{1}};
  \node[fill=black!30,style={circle,draw}] (n2) at (14,0) {\small{2}};
\path
        (n1)  [->]  edge[loop,thick] node {} (n1)
        (n1)  [->]  edge[loop left=90,thick] node {} (n1)
        (n1) [->,thick]  edge[thick] node { } (n2);
 \end{tikzpicture}
  &  
 $(f)$  \begin{tikzpicture}
 [scale=.15,auto=left, node distance=1.5cm, 
 ]
 \node[fill=white,style={circle,draw}] (n1) at (4,0) {\small{1}};
  \node[fill=black!30,style={circle,draw}] (n2) at (14,0) {\small{2}};
\path
        (n1)  [->]  edge[loop left=90,thick] node {} (n1)
        (n1) [->,thick]  edge[bend right=10, thick] node { } (n2)
        (n1) [->,thick]  edge[bend right=-10, thick] node { } (n2);
 \end{tikzpicture}
&
$(g)$ \begin{tikzpicture}
 [scale=.15,auto=left, node distance=1.5cm, 
 ]
 \node[fill=white,style={circle,draw}] (n1) at (4,0) {\small{1}};
  \node[fill=black!30,style={circle,draw}] (n2) at (14,0) {\small{2}};
\path
        (n1) [->,thick]  edge[thick] node { } (n2)
         (n1) [->,thick]  edge[loop left=90, thick] node { } (n1);
\path 
        (n2) [->,dashed, thick]  edge[loop right=90, thick] node { } (n2);
 \end{tikzpicture} 
& 
 $(h)$ 
  \begin{tikzpicture}
 [scale=.15,auto=left, node distance=1.5cm, 
 ]
 \node[fill=white,style={circle,draw}] (n1) at (4,0) {\small{1}};
  \node[fill=black!30,style={circle,draw}] (n2) at (14,0) {\small{2}};
\path
        (n1)  [->]  edge[loop,thick] node {} (n1)
        (n1)  [->]  edge[loop left=90,thick] node {} (n1)
        (n1) [->,thick]  edge[bend right=20, thick] node { } (n2)
        (n1) [->,thick]  edge[bend right=-30, thick] node { } (n2);
 \end{tikzpicture}\\
\\
 $(i)$  \begin{tikzpicture}
 [scale=.15,auto=left, node distance=1.5cm, 
 ]
 \node[fill=white,style={circle,draw}] (n1) at (4,0) {\small{1}};
  \node[fill=black!30,style={circle,draw}] (n2) at (14,0) {\small{2}};
\path
        (n1)  [->]  edge[loop,thick] node {} (n1)
        (n1)  [->]  edge[loop left=90,thick] node {} (n1)
        (n1) [->,thick]  edge[thick] node { } (n2);
\path 
         (n2)  [->, dashed]  edge[loop right=90,thick] node {} (n2); 
 \end{tikzpicture}
&
$(j)$ \begin{tikzpicture}
 [scale=.15,auto=left, node distance=1.5cm, 
 ]
 \node[fill=white,style={circle,draw}] (n1) at (4,0) {\small{1}};
  \node[fill=black!30,style={circle,draw}] (n2) at (14,0) {\small{2}};
\path
        (n1) [->,thick]  edge[bend right=10] node { } (n2);
 \path 
        (n2) [->,dashed, thick]  edge[bend left=-10] node { } (n1);
 \end{tikzpicture}
&
 $(k)$ \begin{tikzpicture}
 [scale=.15,auto=left, node distance=1.5cm, 
 ]
 \node[fill=white,style={circle,draw}] (n1) at (4,0) {\small{1}};
  \node[fill=black!30,style={circle,draw}] (n2) at (14,0) {\small{2}};
\path
        (n1)  [->]  edge[loop left=90,thick] node {} (n1)
        (n1) [->,thick]  edge[bend right=10, thick] node { } (n2);
\path 
        (n2) [->,dashed, thick]  edge[bend left=-10] node { } (n1);
 \end{tikzpicture}
  &   
 $(l)$ \begin{tikzpicture}
 [scale=.15,auto=left, node distance=1.5cm, 
 ]
 \node[fill=white,style={circle,draw}] (n1) at (4,0) {\small{1}};
  \node[fill=black!30,style={circle,draw}] (n2) at (14,0) {\small{2}};
\path
         (n1) [->,thick]  edge[bend right=10, thick] node { } (n2);
\path 
        (n2) [->,thick, dashed]  edge[bend left =-10] node { } (n1)
        (n2) [->,thick, dashed]  edge[bend left=-30] node { } (n1);
 \end{tikzpicture}\\
 \\
 $(m)$   \begin{tikzpicture}
 [scale=.15,auto=left, node distance=1.5cm, 
 ]
 \node[fill=white,style={circle,draw}] (n1) at (4,0) {\small{1}};
  \node[fill=black!30,style={circle,draw}] (n2) at (14,0) {\small{2}};
\path
        (n1)  [->]  edge[loop left=90,thick] node {} (n1)
        (n1) [->,thick]  edge[bend right=10, thick] node { } (n2)
         (n1) [->,thick]  edge[bend right=40, thick] node { } (n2);
\path 
        (n2) [->,thick, dashed]  edge[bend left=-10, thick] node { } (n1);      
 \end{tikzpicture}
 &  
 $(n)$  \begin{tikzpicture}
 [scale=.15,auto=left, node distance=1.5cm, 
 ]
 \node[fill=white,style={circle,draw}] (n1) at (4,0) {\small{1}};
  \node[fill=black!30,style={circle,draw}] (n2) at (14,0) {\small{2}};
\path
        (n1)  [->]  edge[loop left=90,thick] node {} (n1)
         (n1) [->,thick]  edge[bend right=10] node { } (n2);
\path          
         (n2)  [->, dashed]  edge[loop right=90,thick] node {} (n2)       
        (n2) [->,thick, dashed]  edge[bend left=-10] node { } (n1);
 \end{tikzpicture} 
 & 
 $(o)$ \begin{tikzpicture}
 [scale=.15,auto=left, node distance=1.5cm, 
 ]
 \node[fill=white,style={circle,draw}] (n1) at (4,0) {\small{1}};
  \node[fill=black!30,style={circle,draw}] (n2) at (14,0) {\small{2}};
\path
        (n1)  [->]  edge[bend right=10,thick] node {} (n2)
         (n1) [->,thick]  edge[bend right=40, thick] node { } (n2);
 \path 
         (n2)  [->,dashed]  edge[bend left=-10,thick] node {} (n1)      
        (n2) [->,thick,dashed]  edge[bend left=-40, thick] node { } (n1);
 \end{tikzpicture} 
\end{tabular}
\end{center}
}
\caption{Connected 2-node REI networks with input valence $\leq 2$. 
Networks (a)-(i) are in the ODE-class of network NH1 in Figure~\ref{fig:ODE_general_REIV} and networks (j)-(o) are in the ODE-class of network NH2 in Figure~\ref{fig:ODE_general_REIV}.
}
\label{fig:2NCNREIV2}
\end{figure}

\begin{table}
{\tiny 
\begin{tabular}{|l|l|l|}
\hline 
 & & \\
(a) 
$\begin{array}{l}
\dot{x}_1 = f(x^+_1) \\
\dot{x}_2 = g(x^-_2; x^+_1)
\end{array}$
& 
(b) 
$ \begin{array}{l}
\dot{x}_1 = f(x^+_1) \\
\dot{x}_2 = g(x^-_2; \overline{x^+_1,x^+_1})
\end{array}$ 
& 
(c) 
$
\begin{array}{l}
\dot{x}_1 = f(x^+_1) \\
\dot{x}_2 = g(x^-_2; x^+_1;x^-_2)
\end{array}$ \\ 
 & & \\
\hline
 & & \\
(d) 
$\begin{array}{l}
\dot{x}_1 = f(x^ +_1; x^+_1) \\
\dot{x}_2 = g(x^-_2; x^+_1)
\end{array}$ 
 & 
(e) 
$ \begin{array}{l}
\dot{x}_1 = f(x^+_1; \overline{x^+_1,x^+_1}) \\
\dot{x}_2 = g(x^-_2; x^+_1)
\end{array}$ 
& 
(f) 
$\begin{array}{l}
\dot{x}_1 = f(x^+_1; x^+_1); \\
\dot{x}_2 = g(x^-_2; \overline{x^+_1, x^+_1})
\end{array}$ \\
 & & \\
\hline
 & & \\
(g) 
$\begin{array}{l}
\dot{x}_1 = f(x^+_1; x^+_1) \\
\dot{x}_2 = g(x^-_2; x^+_1; x^-_2)
\end{array}$ 
& 
(h) 
$\begin{array}{l}
\dot{x}_1 = f(x^+_1; \overline{x^+_1, x^+_1}) \\
\dot{x}_2 = g(x^-_2; \overline{x^+_1, x^+_1})
\end{array}$
& 
(i) 
$\begin{array}{l}
\dot{x}_1 = f(x^+_1; \overline{x^+_1, x^+_1}) \\
\dot{x}_2 = g(x^-_2; x^+_1; x^-_2)
\end{array}$ \\
 & & \\
\hline
 & & \\
(j) 
$\begin{array}{l}
\dot{x}_1 = f(x^+_1; x^-_2) \\
\dot{x}_2 = g(x^-_2; x^+_1)
\end{array}$ 
& 
(k) 
$\begin{array}{l}
\dot{x}_1 = f(x^+_1; x^+_1; x^-_2) \\
\dot{x}_2 = g(x^-_2; x^+_1)
\end{array}$ 
& 
(l) $
\begin{array}{l}
\dot{x}_1 = f(x^+_1; \overline{x^-_2,x^-_2}) \\
\dot{x}_2 = g(x^-_2; x^+_1)
\end{array}$ 
\\
 & & \\
\hline
 & & \\  
(m) $\begin{array}{l}
\dot{x}_1 = f(x^+_1; x^+_1; x^-_2) \\
\dot{x}_2 = g(x^-_2; \overline{x^+_1,x^+_1})
\end{array}$ 
& 
(n) 
$ \begin{array}{l}
\dot{x}_1 = f(x^+_1; x^+_1; x^-_2) \\
\dot{x}_2 = g(x^-_2; x^+_1; x^-_2)
\end{array}$ 
& 
(o) 
$
\begin{array}{l}
\dot{x}_1 = f(x^+_1; \overline{x^-_2, x^-_2}) \\
\dot{x}_2 = g(x^-_2; \overline{x^+_1, x^+_1})
\end{array}$ \\
 & & \\
\hline  
\end{tabular}
}
\vspace{.2cm}
\caption{
Admissible ODEs for the networks in Figure~\ref{fig:2NCNREIV2}. Since all networks 
 are inhomogeneous, the admissible ODEs are determined by two functions $f$ and $g$. 
} 
\label{table:admissible_2NCNREIV2}
\end{table}

\begin{rems} 

(i) The networks of the ODE-class of NH1 
are feedforward; the networks of the ODE-class of NH2 
are transitive.

(ii) The Smolen network, 
which is network $(n)$ of Figure~\ref{fig:2NCNREIV2}, belongs to the ODE-class of NH2.
There are no ${\mathbb Z}_2$-symmetric networks in this classification.

(iii) For any choice of node state spaces, the space of admissible maps for the network of NH1 
is strictly contained in the  space of admissible maps for the network NH2. 
\hfill $\Diamond$
\end{rems}

\subsection{Connected 2-node PEI Networks} 
\label{S:C2PEI}

A $2$-node PEI network is the network on the left of Figure~\ref{fig:generalREI}, considering the two nodes to be of the same type, for a suitable choice of nonnegative integer arrow multiplicities $\alpha$, $\beta$, $\gamma$, $\delta$. 
See~Figure~\ref{fig:generalPEI}.

\begin{figure}
\begin{tikzpicture}
 [scale=.15,auto=left, node distance=2cm
 ]
 \node[fill=white,style={circle,draw}] (n1) at (4,0) {\small{1}};
  \node[fill=white,style={circle,draw}] (n2) at (20,0) {\small{2}};
  \path
   (7,1)  [->] edge[thick] node {$\beta$}  (17,1)
 (n1)  [line width=1pt,->]  edge[loop left=90,thick] node {$\alpha$} (n1);
 \path 
 (17,-1)  [dashed,->] edge[thick] node {$\gamma$}  (7,-1)
 (n2)  [dashed,->]  edge[loop right=90,thick] node {$\delta$} (n2); 
 \end{tikzpicture} 
  \caption{A $2$-node PEI network. The nonnegative integer arrow multiplicities are $\alpha, \beta, \gamma, \delta$. The network is connected when one of $\beta$ or $\gamma$ is nonzero.} \label{fig:generalPEI}
\end{figure}
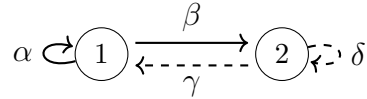

There is  an analogous result to Proposition~\ref{prop_general_val_2REI} for REI networks,
 with the same proof.

\begin{prop} \label{prop_general_val_2REI_dup}
A $2$-node connected PEI network is the network of Figure~{\rm \ref{fig:generalPEI}}, for some choice of nonnegative integer arrow multiplicities $\alpha, \beta, \gamma, \delta$, where one of $\beta$ or $\gamma$ is nonzero. All the statements of Proposition~{\rm \ref{prop_general_val_2REI}} hold, with the additional condition that  the two nodes have the same node-type. 
When the multiplicities satisfy  $\alpha = \beta,\ \gamma = \delta$ then the network is homogeneous and {\rm (\ref{eq:general2REI})}
holds with $f=g$. 
\end{prop}

\subsection{Connected 2-node PEI Networks: ODE-classes}
\label{S:C2PEINODE}

The classification of PEI networks into ODE-classes is different from that of REI networks since a 2-node PEI network, see Figure~\ref{fig:generalPEI}, has 
three network  adjacency matrices: 
\begin{equation}
\label{E:adj_PEI_3types}
\begin{array}{ll} 
\mbox{Node-type $N^E = N^I$: } 
\id_2 = 
\left[
\begin{array}{cc}
1 & 0 \\
0 & 1 
\end{array}
\right];  &  \\ 
\ \\
\mbox{Arrow-type $A^E$: } 
A_3 = 
\left[
\begin{array}{cc}
\alpha & 0 \\
\beta & 0 
\end{array}
\right]; \quad 
\mbox{Arrow-type $A^I$: } 
A_4= 
\left[
\begin{array}{cc}
0& \gamma \\
0 & \delta 
\end{array}
\right];
\end{array}
\end{equation}
where at least $\beta$ or $\gamma$ is nonzero to obtain a connected network. Therefore
there are infinitely  many distinct ODE-classes for 2-node PEI networks:
 
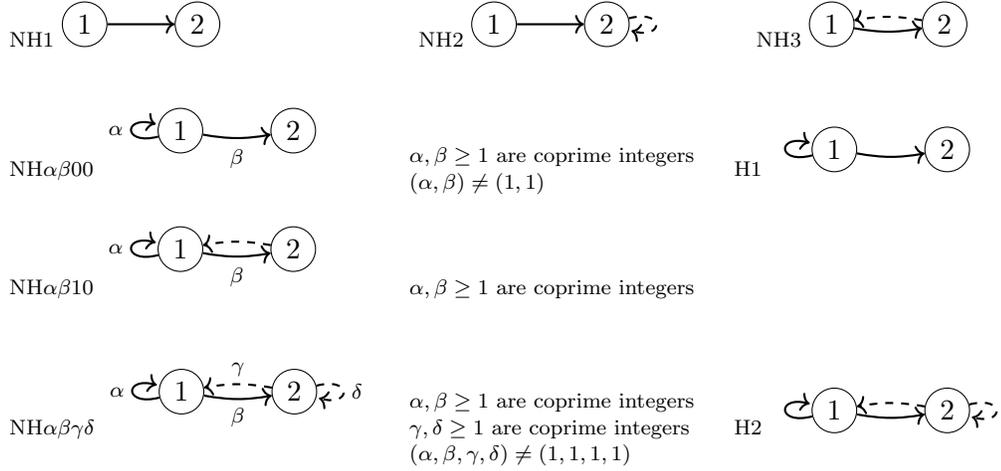
\begin{figure}[!ht]
\begin{center}
{\tiny 
\begin{tabular}{lll} 
NH1 \begin{tikzpicture}
 [scale=.15,auto=left, node distance=1.5cm, 
 ]
 \node[fill=white,style={circle,draw}] (n1) at (4,0) {\small{1}};
  \node[fill=white,style={circle,draw}] (n2) at (14,0) {\small{2}};
 \draw[->, thick] (n1) edge[thick] node {}  (n2); 
 \end{tikzpicture}  
\quad & \quad 
NH2  \begin{tikzpicture}
 [scale=.15,auto=left, node distance=1.5cm, 
 ]
 \node[fill=white,style={circle,draw}] (n1) at (4,0) {\small{1}};
  \node[fill=white,style={circle,draw}] (n2) at (14,0) {\small{2}};
 \draw[->, thick] (n1) edge[thick] node {}  (n2); 
 \draw[->, thick, dashed]    (n2)  edge[loop right=90,thick] node {} (n2); 
 \end{tikzpicture} 
 \quad & \quad 
 NH3 \begin{tikzpicture}
 [scale=.15,auto=left, node distance=1.5cm, 
 ]
 \node[fill=white,style={circle,draw}] (n1) at (4,0) {\small{1}};
  \node[fill=white,style={circle,draw}] (n2) at (14,0) {\small{2}};
\path
        (n1) [->,thick]  edge[bend right=10] node { } (n2);
\path 
        (n2) [->,thick,dashed]  edge[bend left=-10] node { } (n1);
 \end{tikzpicture}  \\
 \\
  \\
 NH$\alpha\beta00$  \begin{tikzpicture}
 [scale=.15,auto=left, node distance=1.5cm]
 \node[fill=white,style={circle,draw}] (n1) at (4,0) {\small{1}};
  \node[fill=white,style={circle,draw}] (n2) at (14,0) {\small{2}};
 \draw[->, thick] (n1) edge[bend right=10] node [below=1pt] {{\tiny $\beta$}}  (n2);
 \draw[->, thick]    (n1)  edge[loop left=90,thick] node {{\tiny $\alpha$}}   (n1); 
 \end{tikzpicture} &     
 \begin{tabular}{l} $ \alpha, \beta \ge 1$ are coprime integers   \\
                             $(\alpha,\beta) \not= (1,1)$
 \end{tabular} & 
 H1  \begin{tikzpicture}
 [scale=.15,auto=left, node distance=1.5cm]
 \node[fill=white,style={circle,draw}] (n1) at (4,0) {\small{1}};
  \node[fill=white,style={circle,draw}] (n2) at (14,0) {\small{2}};
 \draw[->, thick] (n1) edge[bend right=10] node [below=1pt] {}  (n2);
 \draw[->, thick]    (n1)  edge[loop left=90,thick] node {}   (n1); 
 \end{tikzpicture}  \\
 \\
 NH$\alpha\beta10$  \begin{tikzpicture}
 [scale=.15,auto=left, node distance=1.5cm]
 \node[fill=white,style={circle,draw}] (n1) at (4,0) {\small{1}};
  \node[fill=white,style={circle,draw}] (n2) at (14,0) {\small{2}};
 \draw[->, thick] (n1) edge[bend right=10] node [below=1pt] {{\tiny $\beta$}}  (n2);
\draw[->, thick, dashed] (n2) edge[bend left=-10] node [above=1pt]  {}  (n1); 
 \draw[->, thick]    (n1)  edge[loop left=90,thick] node {{\tiny $\alpha$}}   (n1); 
 \end{tikzpicture} &     
 \begin{tabular}{l} $ \alpha, \beta \ge 1$ are coprime integers   \\
 \end{tabular} &  \\
 \\
 \\
NH$\alpha\beta\gamma\delta$  \begin{tikzpicture}
 [scale=.15,auto=left, node distance=1.5cm]
 \node[fill=white,style={circle,draw}] (n1) at (4,0) {\small{1}};
  \node[fill=white,style={circle,draw}] (n2) at (14,0) {\small{2}};
 \draw[->, thick] (n1) edge[bend right=10] node [below=0.1pt] {{\tiny $\beta$}}  (n2);
\draw[->, thick, dashed] (n2) edge[bend left=-10] node [above=0.1pt]  {{\tiny $\gamma$}}  (n1); 
\draw[->, thick]    (n1)  edge[loop left=90,thick] node {{\tiny $\alpha$}}   (n1); 
\draw[->, thick, dashed]    (n2)  edge[loop right=-90,thick] node {{\tiny $\delta$}}   (n2); 
 \end{tikzpicture} & 
 \begin{tabular}{l} $\alpha, \beta \geq 1$ are coprime integers \\
  $\gamma, \delta \geq 1$ are coprime integers \\
  $(\alpha, \beta, \gamma, \delta) \not= (1,1,1,1)$
 \end{tabular}  & 
 H2  \begin{tikzpicture}
 [scale=.15,auto=left, node distance=1.5cm]
 \node[fill=white,style={circle,draw}] (n1) at (4,0) {\small{1}};
  \node[fill=white,style={circle,draw}] (n2) at (14,0) {\small{2}};
 \draw[->, thick] (n1) edge[bend right=10] node [below=0.1pt] {}  (n2);
\draw[->, thick, dashed] (n2) edge[bend left=-10] node [above=0.1pt]  {}  (n1); 
\draw[->, thick]    (n1)  edge[loop left=90,thick] node {}   (n1); 
\draw[->, thick, dashed]    (n2)  edge[loop right=-90,thick] node {}   (n2); 
 \end{tikzpicture} 
\end{tabular}}
\end{center}
\caption{Network representatives, of the ODE-classes of  connected $2$-node PEI networks, up to duality. 
}
\label{fig:ODE_general_PEI}
\end{figure}

\begin{table}[!tb]
{\tiny 
\begin{tabular}{|l|l|}
\hline 
 &   \\
NH1\quad $
 \begin{array}{l}
\dot{x}_1 = f(x_1) \\
\dot{x}_2 = g(x_2; x^+_1)
\end{array}
$ & 
NH2 \quad$
\begin{array}{l}
\dot{x}_1 = f(x_1) \\
\dot{x}_2 = g(x_2; x^+_1,x^-_2)
\end{array}
$ \\
 &   \\
\hline 
 &   \\
NH3 \quad$
\begin{array}{l}
\dot{x}_1 = f(x_1; x^-_2) \\
\dot{x}_2 = g(x_2; x^+_1)
\end{array}
$&  
 NH$\alpha\beta 00$ \quad$
 \begin{array}{l}
\dot{x}_1 = f(x_1; \underbrace{\overline{x^+_1, \ldots, x^+_1}}_{ \alpha}) \\
\dot{x}_2 = g(x_2; \underbrace{\overline{x^+_1, \ldots, x^+_1}}_{ \beta} )
\end{array}
$  \\
 &   \\
\hline
 &   \\ 
NH$\alpha\beta 10$ \quad$
 \begin{array}{l}
\dot{x}_1 = f(x_1; \underbrace{\overline{x^+_1, \ldots, x^+_1}}_{ \alpha}, x^-_2) \\
\dot{x}_2 = g(x_2; \underbrace{\overline{x^+_1, \ldots, x^+_1}}_{ \beta} )
\end{array}
$  
&
NH$\alpha\beta\gamma\delta$ \quad
$

\begin{array}{l}
\dot{x}_1 = f(x_1; \underbrace{\overline{x^+_1, \ldots, x^+_1}}_{ \alpha};     \underbrace{\overline{x^-_2, \ldots, x^-_2}}_{\gamma}) \\
\dot{x}_2 = g(x_2; \underbrace{\overline{x^+_1, \ldots, x^+_1}}_{ \beta};     \underbrace{\overline{x^-_2, \ldots, x^-_2}}_{ \delta})
\end{array}
$ \\
 &   \\
\hline  
 &   \\
H1 $

\begin{array}{l}
\dot{x}_1 = f(x_1; x^+_1); \\
\dot{x}_2 = f(x_2; x^+_1)
\end{array}
$ 
& 
H2 $

\begin{array}{l}
\dot{x}_1 = f(x_1; x^+_1,x^-_2); \\
\dot{x}_2 = f(x_2; x^+_1,x^-_2)
\end{array}
$ 
\\
 &   \\
 \hline 
\end{tabular}}
\vspace{.2cm}
\caption{Admissible ODEs for the networks in Figure~\ref{fig:ODE_general_PEI}.}
\label{table:ODE_general_PEI}
\end{table}

\begin{prop} \label{prop:PEI_ODE}
There is an infinity of ODE-classes of connected $2$-node PEI networks, with
 representatives in Figure~{\rm \ref{fig:ODE_general_PEI}} and associated admissible ODEs in Table~{\rm \ref{table:ODE_general_PEI}}.
\end{prop}

\begin{proof}
Given an ODE-class of 2-node PEI networks, consider a network in that class and  the corresponding adjacency matrices $\id_2, A_3$ and $A_4$ as in 
(\ref{E:adj_PEI_3types}). Since the network is connected, we have that at least one of $\beta$, $\gamma$ is nonzero.
The proof divides into four cases according to how many of the integers $\alpha, \beta, \gamma, \delta$ are zero. \\
(i) If all four integers are positive, then the matrices $\id_2, A_3$ and $A_4$ are linearly independent. Moreover, the matrices of the minimal network of that ODE-class are those obtained from $A_3$, multiplied by $1/\mbox{gcd}(\alpha, \beta)$, and  $A_4$, multiplied by $1/\mbox{gcd}(\gamma, \delta)$.
We obtain the ODE-classes of inhomogeneous networks NH$\alpha\beta\gamma\delta$ in Figure~\ref{fig:ODE_general_PEI}, which are
parametrized by the positive integers $\alpha, \beta, \gamma, \delta$, where $\alpha, \beta$ are coprime, $\gamma, \delta$ are coprime, and
  $(\alpha, \beta, \gamma, \delta) \not= (1,1,1,1)$. If $(\alpha, \beta, \gamma, \delta) = (1,1,1,1)$ then we obtain the homogeneous network 
ODE-class H2 of Figure~\ref{fig:ODE_general_PEI}. \\
(ii) If exactly one of the integers $\alpha, \beta, \gamma, \delta$ is zero, then, up to duality, either $\gamma=0$ or $\delta =0$. In the first case,
$$
\left\langle \id_2, 
\left[
\begin{array}{cc}
\alpha    & 0 \\
\beta & 0
\end{array}
\right], 
\left[
\begin{array}{cc}
0 & 0 \\
0 & \delta
\end{array}
\right]
\rangle =
\left\langle \id_2, 
\left[
\begin{array}{cc}
0    & 0 \\
1 & 0
\end{array}
\right], 
\left[
\begin{array}{cc}
0 & 0 \\
0 & 1
\end{array}
\right]
\right\rangle\, \right.
$$
and every such network is ODE-equivalent to network NH2  in Figure~\ref{fig:ODE_general_PEI}. In the second case, 
$$\left\langle \id_2, 
\left[
\begin{array}{cc}
\alpha    & 0 \\
\beta & 0
\end{array}
\right], 
\left[
\begin{array}{cc}
0 & \gamma \\
0 & 0
\end{array}
\right]
\rangle =
\left\langle \id_2, 
\left[
\begin{array}{cc}
\alpha    & 0 \\
\beta & 0
\end{array}
\right], 
\left[
\begin{array}{cc}
0 & 1 \\
0 & 0
\end{array}
\right]
\right\rangle\, \right. .
$$
Now we obtain the ODE-classes NH$\alpha\beta10$ in Figure~\ref{fig:ODE_general_PEI},
which are parametrized by the positive integers $\alpha, \beta$, where $\alpha, \beta$ are coprime. \\
(iii) Suppose that exactly two of the integers $\alpha, \beta, \gamma, \delta$ are zero. Up to duality,  
either $\gamma = \delta = 0$, or $\gamma = \alpha =0$, or $ \delta= \alpha =0$. 
If $\gamma = \delta = 0$ and 
$(\alpha, \beta) \neq (1,1)$ we obtain the ODE-classes NH$\alpha\beta 00$  in Figure~\ref{fig:ODE_general_PEI}, which are
parametrized by the positive integers $\alpha, \beta$ where $\alpha, \beta$ are coprime. 
If  $\gamma = \delta = 0$ and $(\alpha, \beta) = (1,1)$, we obtain the ODE-class containing the homogeneous network 
H1 in Figure~\ref{fig:ODE_general_PEI}.  If $\gamma = \alpha =0$, then since
$$
\left\langle \id_2, 
\left[
\begin{array}{cc}
0    & 0 \\
\beta & 0
\end{array}
\right], 
\left[
\begin{array}{cc}
0 & 0 \\
0 & \delta
\end{array}
\right]
\right\rangle =
\left\langle \id_2, 
\left[
\begin{array}{cc}
0   & 0 \\
1 & 0
\end{array}
\right], 
\left[
\begin{array}{cc}
0 & 0 \\
0 & 1
\end{array}
\right]
\right\rangle, 
$$
we have networks in the ODE-class of network NH2  in Figure~\ref{fig:ODE_general_PEI}.  

Similarly, if  $ \delta= \alpha =0$, then since
$$
\left\langle \id_2, 
\left[
\begin{array}{cc}
0    & 0 \\
\beta & 0
\end{array}
\right], 
\left[
\begin{array}{cc}
0 & \gamma \\
0 & 0
\end{array}
\right]
\right\rangle =
\left\langle \id_2, 
\left[
\begin{array}{cc}
0   & 0 \\
1 & 0
\end{array}
\right], 
\left[
\begin{array}{cc}
0 & 1 \\
0 & 0
\end{array}
\right]
\right\rangle, 
$$
we have networks at the ODE-class of network NH3  in Figure~\ref{fig:ODE_general_PEI}.  \\
(iv)  Suppose that exactly three of the integers $\alpha, \beta, \gamma, \delta$ are zero. Up to duality, we can consider 
only the cases where $\delta=\gamma=0$ and $\alpha =0$, since the networks are connected. Now 
$$
\left\langle \id_2, 
\left[
\begin{array}{cc}
0    & 0 \\
\beta & 0
\end{array}
\right]
\right\rangle =
\left\langle \id_2, 
\left[
\begin{array}{cc}
0   & 0 \\
1 & 0
\end{array}
\right]
\right\rangle, 
$$
so we have networks in the ODE-class of network NH1 in Figure~\ref{fig:ODE_general_PEI}. 
\end{proof}

\subsection{Connected 2-node PEI Networks with valence up to 2}
\label{S:C2PEINV2}

\begin{prop} 

The set of $2$-node connected PEI networks with node input valence up to $2$ contains $15$ networks, which correspond to the networks in Figure~{\rm \ref{fig:2NCNREIV2}} but assuming the nodes to be of the same type. This set is partitioned into
$9$ ODE-classes: $7$ classes
formed by inhomogeneous networks and $2$ by homogeneous networks, see Figure~{\rm \ref{fig:ODE_2NCNPEIV2}} and Table~{\rm \ref{table:ODE_2NCNPEIV2}}. 
\label{prop:PEI2}
\end{prop}

\begin{proof}
Trivially, the $2$-node connected PEI networks with node input valence up to 2 can be obtained from the $2$-node connected REI networks with node input valence up to 2 (see Figure~\ref{fig:2NCNREIV2}) by considering the nodes to have the same type.
The main result of Dias and Stewart~\cite{DS05} on ODE-equivalence implies that 
there are 
9
ODE-classes of PEI networks: a class containing networks $(a)$ and $(b)$, a class containing networks $(d)$ and $(h)$, a class containing networks $(j)$, $(l)$ and $(o)$, 
a class containing networks $(c)$, $(g)$ and $(i)$.
Each of the remaining
five
networks represents a different ODE-class. See Figure~\ref{fig:ODE_2NCNPEIV2} for  minimal ODE-class representatives. 
Table~\ref{table:ODE_2NCNPEIV2} lists the admissible maps for the networks in Figure~\ref{fig:ODE_2NCNPEIV2}. 
\end{proof}

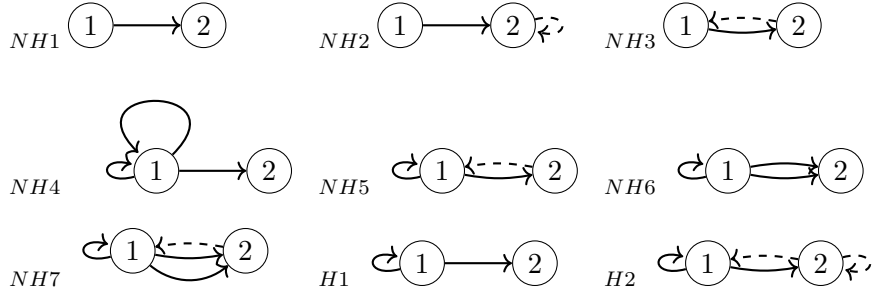
\begin{figure}[!ht]
\tiny{
\begin{center}
\begin{tabular}{lll} 
$NH1$ \begin{tikzpicture}
 [scale=.15,auto=left, node distance=1.5cm, 
 ]
 \node[fill=white,style={circle,draw}] (n1) at (4,0) {\small{1}};
  \node[fill=white,style={circle,draw}] (n2) at (14,0) {\small{2}};
 \draw[->, thick] (n1) edge[thick] node {}  (n2); 
 \end{tikzpicture}  
 &  
$NH2$ \begin{tikzpicture}
 [scale=.15,auto=left, node distance=1.5cm, 
 ]
 \node[fill=white,style={circle,draw}] (n1) at (4,0) {\small{1}};
  \node[fill=white,style={circle,draw}] (n2) at (14,0) {\small{2}};
\path
        (n1) [->,solid, thick]  edge[thick] node { } (n2);
\path        
        (n2) [->,dashed, thick]  edge[loop right=90,thick] node { } (n2);
 \end{tikzpicture} 
  &  
$NH3$ \begin{tikzpicture}
 [scale=.15,auto=left, node distance=1.5cm, 
 ]
 \node[fill=white,style={circle,draw}] (n1) at (4,0) {\small{1}};
  \node[fill=white,style={circle,draw}] (n2) at (14,0) {\small{2}};
\path
        (n1) [->,thick]  edge[bend right=10] node { } (n2);
 \path 
        (n2) [->,dashed, thick]  edge[bend left=-10] node { } (n1);
 \end{tikzpicture} \\
\\
 $NH4$ \begin{tikzpicture}
 [scale=.15,auto=left, node distance=1.5cm, 
 ]
 \node[fill=white,style={circle,draw}] (n1) at (4,0) {\small{1}};
  \node[fill=white,style={circle,draw}] (n2) at (14,0) {\small{2}};
\path
        (n1)  [->]  edge[loop,thick] node {} (n1)
        (n1)  [->]  edge[loop left=90,thick] node {} (n1)
        (n1) [->,thick]  edge[thick] node { } (n2);
 \end{tikzpicture}
  &  
 $NH5$ \begin{tikzpicture}
 [scale=.15,auto=left, node distance=1.5cm, 
 ]
 \node[fill=white,style={circle,draw}] (n1) at (4,0) {\small{1}};
  \node[fill=white,style={circle,draw}] (n2) at (14,0) {\small{2}};
\path
        (n1)  [->]  edge[loop left=90,thick] node {} (n1)
        (n1) [->,thick]  edge[bend right=10, thick] node { } (n2);
\path 
        (n2) [->,dashed, thick]  edge[bend left=-10] node { } (n1);
 \end{tikzpicture}
&   
 $NH6$  \begin{tikzpicture}
 [scale=.15,auto=left, node distance=1.5cm, 
 ]
 \node[fill=white,style={circle,draw}] (n1) at (4,0) {\small{1}};
  \node[fill=white,style={circle,draw}] (n2) at (14,0) {\small{2}};
\path
        (n1)  [->]  edge[loop left=90,thick] node {} (n1)
        (n1) [->,thick]  edge[bend right=10, thick] node { } (n2)
        (n1) [->,thick]  edge[bend right=-10, thick] node { } (n2);
 \end{tikzpicture}\\
 \\  
 $NH7$  \begin{tikzpicture}
 [scale=.15,auto=left, node distance=1.5cm, 
 ]
 \node[fill=white,style={circle,draw}] (n1) at (4,0) {\small{1}};
  \node[fill=white,style={circle,draw}] (n2) at (14,0) {\small{2}};
\path
        (n1)  [->]  edge[loop left=90,thick] node {} (n1)
        (n1) [->,thick]  edge[bend right=10, thick] node { } (n2)
         (n1) [->,thick]  edge[bend right=40, thick] node { } (n2);
\path 
        (n2) [->,thick, dashed]  edge[bend left=-10, thick] node { } (n1);      
 \end{tikzpicture} & 
$H1$ \begin{tikzpicture}
 [scale=.15,auto=left, node distance=1.5cm, 
 ]
 \node[fill=white,style={circle,draw}] (n1) at (4,0) {\small{1}};
  \node[fill=white,style={circle,draw}] (n2) at (14,0) {\small{2}};
\path
        (n1) [->,thick]  edge[thick] node { } (n2)
        (n1) [->,thick]  edge[loop left=90,thick] node { } (n1);
 \end{tikzpicture}
 &  
 $H2$  \begin{tikzpicture}
 [scale=.15,auto=left, node distance=1.5cm, 
 ]
 \node[fill=white,style={circle,draw}] (n1) at (4,0) {\small{1}};
  \node[fill=white,style={circle,draw}] (n2) at (14,0) {\small{2}};
\path
        (n1)  [->]  edge[loop left=90,thick] node {} (n1)
         (n1) [->,thick]  edge[bend right=10] node { } (n2);
\path          
         (n2)  [->, dashed]  edge[loop right=90,thick] node {} (n2)       
        (n2) [->,thick, dashed]  edge[bend left=-10] node { } (n1);
 \end{tikzpicture} 
  \end{tabular}
\end{center}
\caption{There are $9$
ODE-classes of $2$-node connected PEI networks with input valence up to two: $7$ inhomogeneous classes and $2$ homogeneous classes.
Table~\ref{table:ODE_2NCNPEIV2} lists the corresponding admissible maps.}
\label{fig:ODE_2NCNPEIV2}
}
\end{figure}

\begin{table}
\tiny{
\begin{tabular}{|l|l|l|}
\hline 
& &  \\
NH1 $ 
\begin{array}{l}
\dot{x}_1 = f(x_1); \\
\dot{x}_2 = g(x_2; x^+_1)
\end{array}
$ & 
NH2 $
\begin{array}{l}
\dot{x}_1 = f(x_1); \\
\dot{x}_2 = g(x_2; x^+_1,x^-_2)
\end{array}
$ 
& 
NH3 $
\begin{array}{l}
\dot{x}_1 = f(x_1; x^-_2); \\
\dot{x}_2 = g(x_2; x^+_1)
\end{array}
$
 \\
 & & \\
\hline 
& &  \\
NH4 $
\begin{array}{l}
\dot{x}_1 = f(x_1; \overline{x^+_1,x^+_1}); \\
\dot{x}_2 = g(x_2; x^+_1)
\end{array}
$ 
& 
NH5 $
\begin{array}{l}
\dot{x}_1 = f(x_1; x^+_1, x^-_2); \\
\dot{x}_2 = g(x_2; x^+_1)
\end{array}
$ &
NH6 $
\begin{array}{l}
\dot{x}_1 = f(x_1; x^+_1); \\
\dot{x}_2 = g(x_2; \overline{x^+_1, x^+_1})
\end{array}
$ 
\\
 & &  \\
\hline  
& &   \\
NH7 $
\begin{array}{l}
\dot{x}_1 = f(x_1; x^+_1, x^-_2); \\
\dot{x}_2 = g(x_2; \overline{x^+_1,x^+_1})
\end{array}
$
& 
H1 $
\begin{array}{l}
\dot{x}_1 = f(x_1; x^+_1); \\
\dot{x}_2 = f(x_2; x^+_1)
\end{array}
$ 
& 
H2 $
\begin{array}{l}
\dot{x}_1 = f(x_1; x^+_1,x^-_2); \\
\dot{x}_2 = f(x_2; x^+_1,x^-_2)
\end{array}
$ 
\\
 & &  \\
\hline  
\end{tabular}
\caption{Admissible maps for the networks in Figure~\ref{fig:ODE_2NCNPEIV2}.} 
\label{table:ODE_2NCNPEIV2}
}
\end{table}

\subsection{Connected 2-node UEI Networks}
\label{S:C2UEI}

We now consider $2$-node UEI networks. 
Here, node $1$ is excitatory and node $2$ is inhibitory. 
To classify these we split the network into two subnetworks, each containing both nodes. 
One subnetwork retains only the
excitatory arrows, the other retains only the inhibitory ones. We classify each
subnetwork separately and reassemble them.

\begin{figure}
\begin{tabular}{ll}
\begin{tikzpicture}
 [scale=.15,auto=left, node distance=2cm
 ]
 \node[fill=white,style={circle,draw}] (n1) at (4,0) {\small{1}};
  \node[fill=black!30,style={circle,draw}] (n2) at (20,0) {\small{2}};
  \path
   (7,1)  [->] edge[thick] node {$\beta_1$}  (17,1)
 (n1)  [line width=1pt,->]  edge[loop left=90,thick] node {$\alpha_1$} (n1);
 \path 
 (17,-1)  [->] edge[thick] node {$\beta_2$}  (7,-1)
 (n2)  [->]  edge[loop right=90,thick] node {$\alpha_2$} (n2); 
 \end{tikzpicture}  & 
 \begin{tikzpicture}
 [scale=.15,auto=left, node distance=2cm
 ]
 \node[fill=white,style={circle,draw}] (n1) at (4,0) {\small{1}};
  \node[fill=black!30,style={circle,draw}] (n2) at (20,0) {\small{2}};
  \path
   (7,1)  [dashed,->] edge[thick] node {$\gamma_1$}  (17,1)
 (n1)  [line width=1pt,->]  edge[loop left=90,thick] node {$\delta_1$} (n1);
 \path 
 (17,-1)  [dashed, ->] edge[thick] node {$\gamma_2$}  (7,-1)
 (n2)  [->]  edge[loop right=90,thick] node {$\delta_2$} (n2); 
  \end{tikzpicture} 
 \end{tabular}
 \caption{A $2$-node UEI network has two $2$-node subnetworks.} \label{fig:generalUEI}
\end{figure}
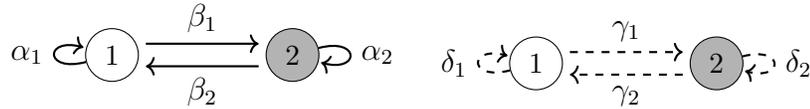

\begin{prop}\label{lem:general_2_node_UEI}
A connected $2$-node UEI network splits into the two subnetworks in  Figure~{\rm \ref{fig:generalUEI}}, for suitable nonnegative integer arrow multiplicities $\alpha_i, \beta_i, \gamma_i, \delta_i$, where $i=1,2$ and at least one of the $\beta_1,\beta_2, \gamma_1, \gamma_2$ is nonzero. The admissible ODEs are:
\begin{equation}
\begin{array}{l}
\dot{x}_1 = f(x^+_1; \overline{ \underbrace{x^+_1, \ldots, x^+_1}_{\alpha_1},\underbrace{x^+_2, \ldots, x^+_2}_{\beta_2} }, 
\overline{ \underbrace{x^-_1, \ldots, x^-_1}_{ \delta_1}, \underbrace{x^-_2, \ldots, x^-_2}_{\gamma_2} }) \\
\\
\dot{x}_2 = g(x^-_2;  \overline{\underbrace{x^+_1, \ldots, x^+_1}_{{ \beta_1}},\underbrace{x^+_2, \ldots, x^+_2}_{\alpha_2}}, \overline{\underbrace{x^-_1, \ldots, x^-_1}_{{ \gamma_1}},\underbrace{x^-_2, \ldots, x^-_2}_{\delta_2}}) \\
\end{array}
\label{eq:general2UEI}
\end{equation}
Here $x^+_1, x^-_2 \in \R^k$, so the total state space is $\R^{2k}$, and $f$ and $g$ 
are smooth functions invariant under permutation of the variables under each overline.  

To simplify notation, temporarily omit the $\pm$ superscripts, and define
\begin{equation} 
\begin{array}{l} 
p =
(
x_1, 
\underbrace{x_1, \ldots, x_1}_{\alpha_1},  
\underbrace{x_2, \ldots, x_2}_{\beta_2},   
\underbrace{x_1, \ldots, x_1}_{\delta_1}, 
\underbrace{x_2, \ldots, x_2}_{\gamma_2}
) \\ 
q= 
(
x_2, 
\underbrace{x_1, \ldots, x_1}_{{\beta_1}}, 
\underbrace{x_2, \ldots, x_2}_{{\alpha_2}}, 
\underbrace{x_1, \ldots, x_1}_{\gamma_1}, 
\underbrace{x_2, \ldots, x_2}_{\delta_2}
)\\
\\
a = \frac{\partial f}{\partial x_1}\left|_{p}\right. \qquad 
b_1 =  \frac{\partial f}{\partial x_2}\left|_{p}\right. \qquad 
b_2 =  \frac{\partial f}{\partial x_{2 + \alpha_1}}\left|_{p}\right. \\
c_1 = \frac{\partial f}{\partial x_{2 + \alpha_1 + \beta_2}}\left|_{p}\right. \qquad 
c_2 = \frac{\partial f}{\partial x_{2 + \alpha_1 + \beta_2 + \delta_1}}\left|_{p}\right. \\
d = \frac{\partial g}{\partial x_1}\left|_{q}\right. \qquad  
e_1 = \frac{\partial g}{\partial x_2}\left|_{q}\right. \qquad 
e_2 = \frac{\partial g}{\partial x_{2 + \beta_1 }}\left|_{q}\right. \\
f_1 = \frac{\partial g}{\partial x_{2 + \beta_1 + \alpha_2}}\left|_{q}\right. \qquad 
f_2 = \frac{\partial g}{\partial x_{2 + \beta_1 + \alpha_2 + \gamma_1}}\left|_{q}\right.\, 
\end{array}
\label{eq:notation_UEI}
\end{equation}
the linearization of \eqref{eq:general2REI} at $(x_1,x_2)$ is the $2k \times 2k$ matrix 
\begin{equation}
\left[
\begin{array}{cc}
a + \alpha_1 b_1 + \delta_1 c_1 & \beta_2 b_2  + \gamma_2 c_2 \\
\beta_1 e_1 +  \gamma_1 f_1 &  d + \alpha_2 e_2 + \delta_2 f_2
\end{array}
\right]\, . 
\label{eq:lin_UEI}
\end{equation}
\end{prop}

\begin{proof}
For a UEI network there is no restriction on the tail node-type for $A^E$ arrows and $A^I$ arrows. The adjacency matrices are therefore
\begin{equation}
\begin{array}{ll} 
\mbox{Node-type $N^E$: } 
A_1 = 
\left[
\begin{array}{cc}
1 & 0 \\
0 & 0 
\end{array}
\right]; &  
\mbox{Node-type $N^I$: } 
A_2 = 
\left[
\begin{array}{cc}
0& 0 \\
0 & 1 
\end{array}
\right]; \\ 
\ \\
\mbox{Arrow-type $A^E$: } 
A_3 = 
\left[
\begin{array}{cc}
\alpha_1 & \beta_2 \\
\beta_1 & \alpha_2 
\end{array}
\right];  & 
\mbox{Arrow-type $A^I$: } 
A_4= 
\left[
\begin{array}{cc}
\delta_1& \gamma_2 \\
\gamma_1 & \delta_2
\end{array}
\right];
\end{array}
\label{eq:gen_adj_UEI}
\end{equation} 
where at least one of the $\beta_i$ or $\gamma_i$ is nonzero to guarantee connectedness. The network is the union of two $2$-node subnetworks, one containing the arrows of type $A^E$ and the other of  type $A^I$, which correspond to the networks in Figure~\ref{fig:generalUEI}.

Concerning the linearization of equations (\ref{eq:general2UEI}) at  $(x_1,x_2)$, the symmetries of the functions $f$ and $g$ imply (\ref{eq:lin_UEI}). 
\end{proof}

\subsection{Connected 2-node UEI Networks: ODE-classes}
\label{S:C2UEIODE}

\begin{prop} \label{prop:UEI_ODE_classes}
There is an infinity of ODE-classes of connected $2$-node UEI networks, with
 representatives in Figure~{\rm \ref{fig:ODE_general_UEIV}} and associated admissible ODEs in Table~{\rm \ref{table:ODE_general_UEIV}}. 
\end{prop}

\begin{figure}[!ht]
\begin{center}
{\tiny 
\begin{tabular}{ll} 
$NH1$ \begin{tikzpicture}
 [scale=.15,auto=left, node distance=1.5cm, 
 ]
 \node[fill=white,style={circle,draw}] (n1) at (4,0) {\small{1}};
  \node[fill=black!30,style={circle,draw}] (n2) at (14,0) {\small{2}};
 \draw[->, thick] (n1) edge[thick] node {}  (n2); 
 \end{tikzpicture}  
\quad & \quad 
$NH2$ \begin{tikzpicture}
 [scale=.15,auto=left, node distance=1.5cm, 
 ]
 \node[fill=white,style={circle,draw}] (n1) at (4,0) {\small{1}};
  \node[fill=black!30,style={circle,draw}] (n2) at (14,0) {\small{2}};
\path
        (n1) [->,thick]  edge[bend right=10] node { } (n2);
\path 
        (n2) [->,thick,dashed]  edge[bend left=-10] node { } (n1);
 \end{tikzpicture}  \\
  \\
 $NH\beta_1\beta_2$  \begin{tikzpicture}
 [scale=.15,auto=left, node distance=1.5cm]
 \node[fill=white,style={circle,draw}] (n1) at (4,0) {\small{1}};
  \node[fill=black!30,style={circle,draw}] (n2) at (14,0) {\small{2}};
 \draw[->, thick] (n1) edge[bend right=10] node [below=1pt] {{\tiny $\beta_1$}}  (n2);
\draw[->, thick] (n2) edge[bend left=-10] node [above=1pt]  {{\tiny $\beta_2$}}  (n1); 
 \end{tikzpicture} &     
 \begin{tabular}{l} $ \beta_1, \beta_2 \ge 1$ are coprime integers \\
 \end{tabular} 
\end{tabular}}
\end{center}
\caption{Network representatives, of the ODE-classes of  connected $2$-node UEI networks, up to duality. 
The representatives NH1 and NH2 are REI networks.
}
\label{fig:ODE_general_UEIV}
\end{figure}
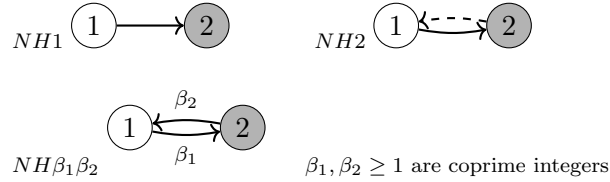

\begin{table}[!ht]
{\tiny 
\begin{tabular}{|l|l|}
\hline 
 & \\
NH1\quad $
 \begin{array}{l}
\dot{x}_1 = f(x^+_1) \\
\dot{x}_2 = g(x^-_2; x^+_1)
\end{array}
$ & 
NH2 \quad$
 \begin{array}{l}
\dot{x}_1 = f(x^+_1; x^-_2) \\
\dot{x}_2 = g(x^-_2; x^+_1)
\end{array}
$ 
\\
 & \\
 \hline 
 & \\
 NH$\beta_1\beta_2$ \quad$
\begin{array}{l}
\dot{x}_1 = f(x^+_1; \underbrace{\overline{x^+_2, \ldots, x^+_2}}_{ \beta_2}) \\
\dot{x}_2 = g(x^-_2; \underbrace{\overline{x^+_1, \ldots, x^+_1}}_{ \beta_1} )
\end{array}
$  
& \\
& \\
\hline  
\end{tabular}
}
\vspace{.2cm}
\caption{Admissible ODEs for the networks in Figure~\ref{fig:ODE_general_UEIV}.}
\label{table:ODE_general_UEIV}
\end{table}

\begin{proof}
Clearly
$$
\langle A_1,A_2,A_3,A_4 \rangle  = 
\left\langle A_1, A_2, 
\left[
\begin{array}{cc}
0    & \beta_2 \\
\beta_1 & 0
\end{array}
\right], 
\left[
\begin{array}{cc}
0 & \gamma_2 \\
\gamma_1 & 0
\end{array}
\right]
\right\rangle , 
$$
where the $A_i$ are given by (\ref{eq:gen_adj_UEI}). \\
(i) If $\beta_2=0$ and $\beta_1 \not= 0$ (duality deals with the case $\beta_2 \not=0$ and $\beta_1 =0$) then 
$$
\langle A_1,A_2,A_3,A_4\rangle  = 
\left\langle A_1, A_2, 
\left[
\begin{array}{cc}
0    & 0 \\
1 & 0
\end{array}
\right], 
\left[
\begin{array}{cc}
0 & \gamma_2 \\
0 & 0
\end{array}
\right]
\right\rangle \, .
$$
(i.a) If $\gamma_2 = 0$ then 
$$
\langle A_1,A_2,A_3,A_4\rangle  = 
\left\langle A_1, A_2, 
\left[
\begin{array}{cc}
0    & 0 \\
1 & 0
\end{array}
\right]\right\rangle ,
$$
which corresponds to network NH1 in Figure~\ref{fig:ODE_general_UEIV}. \\
(i.b) If $\gamma_2 \not=0$ we get network NH2 in Figure~\ref{fig:ODE_general_UEIV} with 
adjacency matrices 
$$
A_1, A_2, 
\left[
\begin{array}{cc}
0    & 0 \\
1 & 0
\end{array}
\right], 
\left[
\begin{array}{cc}
0 & 1 \\
0 & 0
\end{array}
\right]\, .
$$
The linear space
$$
\left\langle A_1, A_2, 
\left[
\begin{array}{cc}
0    & 0 \\
1 & 0
\end{array}
\right], 
\left[
\begin{array}{cc}
0 & 1 \\
0 & 0
\end{array}
\right]
\right\rangle 
$$
is 4-dimensional; that is, it coincides with the linear space of all $2 \times 2$-matrices with real entries.\\
(ii) If both $\beta_1, \beta_2$ are nonzero, we distinguish two cases:\\
(ii.a) When $\gamma_1 = \gamma_2 = 0$ we obtain networks with adjacency matrices 
$$
A_1, A_2, 
\left[
\begin{array}{cc}
0 & \beta_2 \\
\beta_1 & 0
\end{array}
\right] \, .
$$
In particular, we obtain the networks NH3 (if  $\beta_1 = \beta_2 =1$) and NH4 (if  $\beta_1 = 2, \beta_2 =1$) pictured in Figure~\ref{fig:ODE_2NCNUEIV2}.  
There is an infinite number of ODE-equivalence classes, with representatives the networks for each  arrow-type adjacency matrix
$$
\left[
\begin{array}{cc}
0 & \beta_2 \\
\beta_1 & 0
\end{array}
\right]
$$ 
where 
$\beta_1, \beta_2 \ge 1$ are coprime, yielding the network 
NH$\beta_1\beta_2$.  \\
(ii.b) If at least one of the $\gamma_1$ or $\gamma_2$ is nonzero, either  
$\langle A_1,A_2,A_3,A_4\rangle = \langle A_1,A_2,A_3\rangle$, leading to one of the previous cases, or 
$\dim \langle A_1,A_2,A_3,A_4\rangle =4$, and we have a network ODE-equivalent to NH2.
This was obtained previously, since  
$$
\begin{array}{rcl} 
 \langle A_1,A_2,A_3,A_4\rangle &=&  
\left\langle A_1, A_2, 
\left[
\begin{array}{cc}
0    & 0 \\
1 & 0
\end{array}
\right], 
\left[
\begin{array}{cc}
0 & 1 \\
0 & 0
\end{array}
\right]
\right\rangle  \, .
\end{array}
$$
(iii) In the final case when $\beta_1 = \beta_2 =0$ we obtain the dual cases of the networks in (ii.a). 
\end{proof}

\subsection{Connected 2-node UEI Networks with valence up to 2}
\label{S:C2UEIV2}

Using the results in the previous section, we list
 all the connected $2$-node UEI networks with valence $\leq 2$.

Proposition~\ref{prop:UEI_ODE_classes} implies:
\begin{prop}
\label{L:3.8}
There are $4$ ODE-classes of connected $2$-node UEI networks with valence $\leq 2$. Representatives are in  Figure~{\rm \ref{fig:ODE_2NCNUEIV2}}, and associated admissible ODEs 
are listed in Table~{\rm \ref{table:ODE_2NCNUEIV2}}. 
\end{prop}
\begin{proof}
Consider the ODE-classes of the connected $2$-node UEI networks given by Proposition~\ref{prop:UEI_ODE_classes}, with representatives in Figure~\ref{fig:ODE_general_UEIV}. These representatives are minimal, in the sense of having a minimum number of arrows. Thus it is enough to check which of these ODE-classes has a minimal representative with valence $\leq 2$.

Clearly NH1 and NH2 in Figure~\ref{fig:ODE_general_UEIV} are representatives of two of the ODE-classes of connected $2$-node UEI networks with valence $\leq 2$.

There are only  two ODE-class representatives obtained from NH$\beta_1\beta_2$ in Figure~\ref{fig:ODE_general_UEIV}, one with $\beta_1 = \beta_2 =1$ and the other with $\beta_1 = 2, \beta_2 =1$. These are networks NH3 and NH4 in Figure~\ref{fig:ODE_2NCNUEIV2}.
\end{proof}

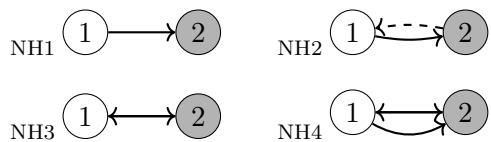
\begin{figure}
\begin{center}
{\tiny 
\begin{tabular}{ll} 
NH1 \begin{tikzpicture}
 [scale=.15,auto=left, node distance=1.5cm, 
 ]
 \node[fill=white,style={circle,draw}] (n1) at (4,0) {\small{1}};
  \node[fill=black!30,style={circle,draw}] (n2) at (14,0) {\small{2}};
 \draw[->, thick] (n1) edge[thick] node {}  (n2); 
 \end{tikzpicture}  
\quad & \quad 
NH2 \begin{tikzpicture}
 [scale=.15,auto=left, node distance=1.5cm, 
 ]
 \node[fill=white,style={circle,draw}] (n1) at (4,0) {\small{1}};
  \node[fill=black!30,style={circle,draw}] (n2) at (14,0) {\small{2}};
\path
        (n1) [->,thick]  edge[bend right=10] node { } (n2);
\path 
        (n2) [->,thick,dashed]  edge[bend left=-10] node { } (n1);
 \end{tikzpicture}   \\
  \\
 NH3
 \begin{tikzpicture}
 [scale=.15,auto=left, node distance=1.5cm]
 \node[fill=white,style={circle,draw}] (n1) at (4,0) {\small{1}};
  \node[fill=black!30,style={circle,draw}] (n2) at (14,0) {\small{2}};
 \draw[->, thick] (n1) edge[thick] node {}  (n2);
\draw[->, thick] (n2) edge[thick] node {}  (n1); 
 \end{tikzpicture}  
\quad & \quad 
 NH4
 \begin{tikzpicture}
[scale=.15,auto=left, node distance=1.5cm]
 \node[fill=white,style={circle,draw}] (n1) at (4,0) {\small{1}};
  \node[fill=black!30,style={circle,draw}] (n2) at (14,0) {\small{2}};
   \draw[->, thick] (n1) edge[thick] node {}  (n2);
\draw[->, thick] (n2) edge[thick] node {}  (n1);  
 \path 
             (n1) [->,thick]  edge[bend left=-30] node { } (n2);
 \end{tikzpicture} 
\end{tabular}}
\end{center}
\caption{
Minimal network representatives of the $4$ ODE-classes of connected $2$-node UEI networks with input valence $\leq 2$. 
The representatives NH1 and NH2 are REI networks.
}
\label{fig:ODE_2NCNUEIV2}
\end{figure}

\begin{table}[h!]
{\tiny 
\begin{tabular}{|l|l|}
\hline 
 & \\
NH1 \quad $
\begin{array}{l}
\dot{x}_1 = f(x^+_1) \\
\dot{x}_2 = g(x^-_2; x^+_1)
\end{array}
$ & 
NH2 \quad$
\begin{array}{l}
\dot{x}_1 = f(x^+_1; x^-_2) \\
\dot{x}_2 = g(x^-_2; x^+_1)
\end{array}
$  \\
 & \\
\hline 
 & \\
NH3\quad
$
\begin{array}{l}
\dot{x}_1 = f(x^+_1; x^+_2) \\
\dot{x}_2 = g(x^-_2; x^+_1)
\end{array}
$ & 
NH4\quad
$ 
\begin{array}{l}
\dot{x}_1 = f(x^+_1; x^+_2) \\
\dot{x}_2 = g(x^-_2; \overline{x^+_1,x^+_1})
\end{array}
$ \\
 & \\
\hline 
\end{tabular}
}
\vspace{.2cm}
\caption{Admissible ODEs for the networks in Figure~\ref{fig:ODE_2NCNUEIV2}.} 
\label{table:ODE_2NCNUEIV2}
\end{table}

\begin{figure}[]
{\tiny 
\begin{center}
\begin{tabular}{llll} 
 $(b.1)$ \begin{tikzpicture}
 [scale=.15,auto=left, node distance=1.5cm, 
 ]
 \node[fill=white,style={circle,draw}] (n1) at (4,0) {\small{1}};
  \node[fill=black!30,style={circle,draw}] (n2) at (14,0) {\small{2}};
 \draw[->, thick] (n1) edge[bend right=10, thick] node {}  (n2); 
 \draw[->, thick,dashed] (n1) edge[bend right=-10,thick] node {}  (n2); 
 \end{tikzpicture}   
&
 $(c)$ \begin{tikzpicture}
 [scale=.15,auto=left, node distance=1.5cm, 
 ]
 \node[fill=white,style={circle,draw}] (n1) at (4,0) {\small{1}};
  \node[fill=black!30,style={circle,draw}] (n2) at (14,0) {\small{2}};
\path
        (n1) [->,solid, thick]  edge[thick] node { } (n2);
\path        
        (n2) [->,thick]  edge[loop right=90,thick] node { } (n2);
 \end{tikzpicture} 
 &
 $(d.1)$ \begin{tikzpicture}
 [scale=.15,auto=left, node distance=1.5cm, 
 ]
 \node[fill=white,style={circle,draw}] (n1) at (4,0) {\small{1}};
  \node[fill=black!30,style={circle,draw}] (n2) at (14,0) {\small{2}};
\path
        (n1) [->,thick]  edge[thick] node { } (n2);
 \path 
        (n1) [->,thick,dashed]  edge[loop left=90,thick] node { } (n1);
 \end{tikzpicture} & 
 $(e.1)$ \begin{tikzpicture}
 [scale=.15,auto=left, node distance=1.5cm, 
 ]
 \node[fill=white,style={circle,draw}] (n1) at (4,0) {\small{1}};
  \node[fill=black!30,style={circle,draw}] (n2) at (14,0) {\small{2}};
\path
        (n1)  [->]  edge[loop,thick] node {} (n1)
        (n1)  [->]  edge[loop left=90,thick] node {} (n1);
 \path 
        (n1) [->,dashed]  edge[thick] node { } (n2);
 \end{tikzpicture} \\
\\
 $(e.2)$ \begin{tikzpicture}
 [scale=.15,auto=left, node distance=1.5cm, 
 ]
 \node[fill=white,style={circle,draw}] (n1) at (4,0) {\small{1}};
  \node[fill=black!30,style={circle,draw}] (n2) at (14,0) {\small{2}};
\path
        (n1)  [->]  edge[loop,thick] node {} (n1);
 \path 
        (n1)  [->, dashed]  edge[loop left=90,thick] node {} (n1);
  \path 
        (n1) [->,thick]  edge[thick] node { } (n2);
 \end{tikzpicture}
 &
 $(f.1)$  \begin{tikzpicture}
 [scale=.15,auto=left, node distance=1.5cm, 
 ]
 \node[fill=white,style={circle,draw}] (n1) at (4,0) {\small{1}};
  \node[fill=black!30,style={circle,draw}] (n2) at (14,0) {\small{2}};
\path
        (n1)  [->,dashed]  edge[loop left=90,thick] node {} (n1);
\path 
        (n1) [->,thick]  edge[bend right=10, thick] node { } (n2)
        (n1) [->,thick]  edge[bend right=-10, thick] node { } (n2);
 \end{tikzpicture} 
&
 $(f.2)$  \begin{tikzpicture}
 [scale=.15,auto=left, node distance=1.5cm, 
 ]
 \node[fill=white,style={circle,draw}] (n1) at (4,0) {\small{1}};
  \node[fill=black!30,style={circle,draw}] (n2) at (14,0) {\small{2}};
\path
        (n1)  [->]  edge[loop left=90,thick] node {} (n1)
        (n1) [->,thick]  edge[bend right=10, thick] node { } (n2);
  \path 
        (n1) [->,dashed]  edge[bend right=-10, thick] node { } (n2);
 \end{tikzpicture} 
&
$(g)$ \begin{tikzpicture}
 [scale=.15,auto=left, node distance=1.5cm, 
 ]
 \node[fill=white,style={circle,draw}] (n1) at (4,0) {\small{1}};
  \node[fill=black!30,style={circle,draw}] (n2) at (14,0) {\small{2}};
\path
        (n1) [->,thick]  edge[thick] node { } (n2)
         (n1) [->,thick]  edge[loop left=90, thick] node { } (n1);
\path 
        (n2) [->,thick]  edge[loop right=90, thick] node { } (n2);
 \end{tikzpicture} \\
 \\
 $(g.1)$ \begin{tikzpicture}
 [scale=.15,auto=left, node distance=1.5cm, 
 ]
 \node[fill=white,style={circle,draw}] (n1) at (4,0) {\small{1}};
  \node[fill=black!30,style={circle,draw}] (n2) at (14,0) {\small{2}};
\path
        (n1) [->,thick]  edge[thick] node { } (n2);
 \path 
         (n1) [->,dashed]  edge[loop left=90, thick] node { } (n1);
\path 
        (n2) [->,thick]  edge[loop right=90, thick] node { } (n2);
 \end{tikzpicture}  &
 $(g.2)$ \begin{tikzpicture}
 [scale=.15,auto=left, node distance=1.5cm, 
 ]
 \node[fill=white,style={circle,draw}] (n1) at (4,0) {\small{1}};
  \node[fill=black!30,style={circle,draw}] (n2) at (14,0) {\small{2}};
\path
        (n1) [->,dashed]  edge[thick] node { } (n2);
\path 
         (n1) [->,thick]  edge[loop left=90, thick] node { } (n1);
\path 
        (n2) [->,thick]  edge[loop right=90, thick] node { } (n2);
 \end{tikzpicture}  
&
 $(h.1)$ 
  \begin{tikzpicture}
 [scale=.15,auto=left, node distance=1.5cm, 
 ]
 \node[fill=white,style={circle,draw}] (n1) at (4,0) {\small{1}};
  \node[fill=black!30,style={circle,draw}] (n2) at (14,0) {\small{2}};
\path
        (n1)  [->]  edge[loop,thick] node {} (n1);
 \path 
        (n1)  [->,dashed]  edge[loop left=90,thick] node {} (n1);
\path 
        (n1) [->,thick]  edge[bend right=20, thick] node { } (n2)
        (n1) [->,thick]  edge[bend right=-30, thick] node { } (n2);
 \end{tikzpicture} 
 &
 $(h.2)$ 
  \begin{tikzpicture}
 [scale=.15,auto=left, node distance=1.5cm, 
 ]
 \node[fill=white,style={circle,draw}] (n1) at (4,0) {\small{1}};
  \node[fill=black!30,style={circle,draw}] (n2) at (14,0) {\small{2}};
\path
        (n1)  [->,dashed]  edge[loop,thick] node {} (n1);
 \path 
        (n1)  [->,dashed]  edge[loop left=90,thick] node {} (n1);
\path 
        (n1) [->,thick]  edge[bend right=20, thick] node { } (n2)
        (n1) [->,thick]  edge[bend right=-30, thick] node { } (n2);
 \end{tikzpicture} \\
\\
 $(h.3)$ 
  \begin{tikzpicture}
 [scale=.15,auto=left, node distance=1.5cm, 
 ]
 \node[fill=white,style={circle,draw}] (n1) at (4,0) {\small{1}};
  \node[fill=black!30,style={circle,draw}] (n2) at (14,0) {\small{2}};
\path
        (n1)  [->]  edge[loop,thick] node {} (n1);
 \path 
        (n1)  [->]  edge[loop left=90,thick] node {} (n1);
\path 
        (n1) [->,dashed]  edge[bend right=20, thick] node { } (n2);
\path 
        (n1) [->,thick]  edge[bend right=-30, thick] node { } (n2);
 \end{tikzpicture} 
&
 $(h.4)$ 
  \begin{tikzpicture}
 [scale=.15,auto=left, node distance=1.5cm, 
 ]
 \node[fill=white,style={circle,draw}] (n1) at (4,0) {\small{1}};
  \node[fill=black!30,style={circle,draw}] (n2) at (14,0) {\small{2}};
\path
        (n1)  [->]  edge[loop,thick] node {} (n1);
 \path 
        (n1)  [->,dashed]  edge[loop left=90,thick] node {} (n1);
\path 
        (n1) [->,dashed]  edge[bend right=20, thick] node { } (n2);
 \path 
        (n1) [->,thick]  edge[bend right=-30, thick] node { } (n2);
 \end{tikzpicture}  
&
 $(i)$  \begin{tikzpicture}
 [scale=.15,auto=left, node distance=1.5cm, 
 ]
 \node[fill=white,style={circle,draw}] (n1) at (4,0) {\small{1}};
  \node[fill=black!30,style={circle,draw}] (n2) at (14,0) {\small{2}};
\path
        (n1)  [->]  edge[loop,thick] node {} (n1)
        (n1)  [->]  edge[loop left=90,thick] node {} (n1)
        (n1) [->,thick]  edge[thick] node { } (n2);
\path 
         (n2)  [->]  edge[loop right=90,thick] node {} (n2); 
 \end{tikzpicture}
 &
  $(i.1)$  \begin{tikzpicture}
 [scale=.15,auto=left, node distance=1.5cm, 
 ]
 \node[fill=white,style={circle,draw}] (n1) at (4,0) {\small{1}};
  \node[fill=black!30,style={circle,draw}] (n2) at (14,0) {\small{2}};
\path
        (n1)  [->]  edge[loop,thick] node {} (n1);
\path 
        (n1)  [->,dashed]  edge[loop left=90,thick] node {} (n1);
 \path 
        (n1) [->,thick]  edge[thick] node { } (n2);
\path 
         (n2)  [->]  edge[loop right=90,thick] node {} (n2); 
 \end{tikzpicture} \\
\\
  $(i.2)$  \begin{tikzpicture}
 [scale=.15,auto=left, node distance=1.5cm, 
 ]
 \node[fill=white,style={circle,draw}] (n1) at (4,0) {\small{1}};
  \node[fill=black!30,style={circle,draw}] (n2) at (14,0) {\small{2}};
\path
        (n1)  [->,dashed]  edge[loop,thick] node {} (n1)
        (n1)  [->,dashed]  edge[loop left=90,thick] node {} (n1);
 \path 
        (n1) [->,thick]  edge[thick] node { } (n2);
\path 
         (n2)  [->]  edge[loop right=90,thick] node {} (n2); 
 \end{tikzpicture} 
&
  $(i.3)$  \begin{tikzpicture}
 [scale=.15,auto=left, node distance=1.5cm, 
 ]
 \node[fill=white,style={circle,draw}] (n1) at (4,0) {\small{1}};
  \node[fill=black!30,style={circle,draw}] (n2) at (14,0) {\small{2}};
\path
        (n1)  [->]  edge[loop,thick] node {} (n1)
        (n1)  [->]  edge[loop left=90,thick] node {} (n1);
  \path 
        (n1) [->,dashed]  edge[thick] node { } (n2);
\path 
         (n2)  [->]  edge[loop right=90,thick] node {} (n2); 
 \end{tikzpicture} 
&
  $(i.4)$  \begin{tikzpicture}
 [scale=.15,auto=left, node distance=1.5cm, 
 ]
 \node[fill=white,style={circle,draw}] (n1) at (4,0) {\small{1}};
  \node[fill=black!30,style={circle,draw}] (n2) at (14,0) {\small{2}};
\path
        (n1)  [->]  edge[loop,thick] node {} (n1);
 \path 
        (n1)  [->,dashed]  edge[loop left=90,thick] node {} (n1);
 \path 
        (n1) [->,dashed]  edge[thick] node { } (n2);
\path 
         (n2)  [->]  edge[loop right=90,thick] node {} (n2); 
 \end{tikzpicture} 
&
$(j)$ \begin{tikzpicture}
 [scale=.15,auto=left, node distance=1.5cm, 
 ]
 \node[fill=white,style={circle,draw}] (n1) at (4,0) {\small{1}};
  \node[fill=black!30,style={circle,draw}] (n2) at (14,0) {\small{2}};
\path
        (n1) [->,thick]  edge[bend right=10] node { } (n2);
 \path 
        (n2) [->,thick]  edge[bend left=-10] node { } (n1);
 \end{tikzpicture} \\
 \\
 $(k)$ \begin{tikzpicture}
 [scale=.15,auto=left, node distance=1.5cm, 
 ]
 \node[fill=white,style={circle,draw}] (n1) at (4,0) {\small{1}};
  \node[fill=black!30,style={circle,draw}] (n2) at (14,0) {\small{2}};
\path
        (n1)  [->]  edge[loop left=90,thick] node {} (n1)
        (n1) [->,thick]  edge[bend right=10, thick] node { } (n2);
\path 
        (n2) [->,thick]  edge[bend left=-10] node { } (n1);
 \end{tikzpicture}
&   
 $(k.1)$ \begin{tikzpicture}
 [scale=.15,auto=left, node distance=1.5cm, 
 ]
 \node[fill=white,style={circle,draw}] (n1) at (4,0) {\small{1}};
  \node[fill=black!30,style={circle,draw}] (n2) at (14,0) {\small{2}};
\path
        (n1)  [->]  edge[loop left=90,thick] node {} (n1);
 \path 
        (n1) [->,dashed]  edge[bend right=10, thick] node { } (n2);
\path 
        (n2) [->,thick]  edge[bend left=-10] node { } (n1);
 \end{tikzpicture} 
&
$(k.2)$ \begin{tikzpicture}
 [scale=.15,auto=left, node distance=1.5cm, 
 ]
 \node[fill=white,style={circle,draw}] (n1) at (4,0) {\small{1}};
  \node[fill=black!30,style={circle,draw}] (n2) at (14,0) {\small{2}};
\path
        (n1)  [->,dashed]  edge[loop left=90,thick] node {} (n1);
 \path 
        (n1) [->,thick]  edge[bend right=10, thick] node { } (n2);
\path 
        (n2) [->,thick]  edge[bend left=-10] node { } (n1);
 \end{tikzpicture}
&   
  $(l)$ \begin{tikzpicture}
 [scale=.15,auto=left, node distance=1.5cm, 
 ]
 \node[fill=white,style={circle,draw}] (n1) at (4,0) {\small{1}};
  \node[fill=black!30,style={circle,draw}] (n2) at (14,0) {\small{2}};
\path
         (n1) [->,thick]  edge[bend right=10, thick] node { } (n2);
\path 
        (n2) [->,thick]  edge[bend left =-10] node { } (n1)
        (n2) [->,thick]  edge[bend left=-30] node { } (n1);
 \end{tikzpicture} \\
 \\
 $(l.1)$ \begin{tikzpicture}
 [scale=.15,auto=left, node distance=1.5cm, 
 ]
 \node[fill=white,style={circle,draw}] (n1) at (4,0) {\small{1}};
  \node[fill=black!30,style={circle,draw}] (n2) at (14,0) {\small{2}};
\path
         (n1) [->,thick]  edge[bend right=10, thick] node { } (n2);
\path 
        (n2) [->,thick, dashed]  edge[bend left =-10] node { } (n1);
 \path 
        (n2) [->,thick]  edge[bend left=-30] node { } (n1);
 \end{tikzpicture} 
&
 $(m)$   \begin{tikzpicture}
 [scale=.15,auto=left, node distance=1.5cm, 
 ]
 \node[fill=white,style={circle,draw}] (n1) at (4,0) {\small{1}};
  \node[fill=black!30,style={circle,draw}] (n2) at (14,0) {\small{2}};
\path
        (n1)  [->]  edge[loop left=90,thick] node {} (n1)
        (n1) [->,thick]  edge[bend right=10, thick] node { } (n2)
         (n1) [->,thick]  edge[bend right=40, thick] node { } (n2);
\path 
        (n2) [->,thick]  edge[bend left=-10, thick] node { } (n1);      
 \end{tikzpicture}
 &  
  $(m.1)$   \begin{tikzpicture}
 [scale=.15,auto=left, node distance=1.5cm, 
 ]
 \node[fill=white,style={circle,draw}] (n1) at (4,0) {\small{1}};
  \node[fill=black!30,style={circle,draw}] (n2) at (14,0) {\small{2}};
\path
        (n1)  [->,dashed]  edge[loop left=90,thick] node {} (n1);
\path 
        (n1) [->,thick]  edge[bend right=10, thick] node { } (n2)
         (n1) [->,thick]  edge[bend right=40, thick] node { } (n2);
\path 
        (n2) [->,thick]  edge[bend left=-10, thick] node { } (n1);      
 \end{tikzpicture}
 &  
  $(m.3)$   \begin{tikzpicture}
 [scale=.15,auto=left, node distance=1.5cm, 
 ]
 \node[fill=white,style={circle,draw}] (n1) at (4,0) {\small{1}};
  \node[fill=black!30,style={circle,draw}] (n2) at (14,0) {\small{2}};
\path
        (n1)  [->,dashed]  edge[loop left=90,thick] node {} (n1);
 \path 
        (n1) [->,thick]  edge[bend right=10, thick] node { } (n2)
         (n1) [->,thick]  edge[bend right=40, thick] node { } (n2);
\path 
        (n2) [->,thick, dashed]  edge[bend left=-10, thick] node { } (n1);      
 \end{tikzpicture} \\
 \\
  $(m.4)$   \begin{tikzpicture}
 [scale=.15,auto=left, node distance=1.5cm, 
 ]
 \node[fill=white,style={circle,draw}] (n1) at (4,0) {\small{1}};
  \node[fill=black!30,style={circle,draw}] (n2) at (14,0) {\small{2}};
\path
        (n1)  [->]  edge[loop left=90,thick] node {} (n1)
        (n1) [->,thick]  edge[bend right=10, thick] node { } (n2);
\path 
         (n1) [->,dashed]  edge[bend right=40, thick] node { } (n2);
\path 
        (n2) [->,thick]  edge[bend left=-10, thick] node { } (n1);      
 \end{tikzpicture}
 &  
 $(m.5)$   \begin{tikzpicture}
 [scale=.15,auto=left, node distance=1.5cm, 
 ]
 \node[fill=white,style={circle,draw}] (n1) at (4,0) {\small{1}};
  \node[fill=black!30,style={circle,draw}] (n2) at (14,0) {\small{2}};
\path
        (n1)  [->,dashed]  edge[loop left=90,thick] node {} (n1);
 \path 
        (n1) [->,thick]  edge[bend right=10, thick] node { } (n2);
 \path 
         (n1) [->,dashed]  edge[bend right=40, thick] node { } (n2);
\path 
        (n2) [->,thick]  edge[bend left=-10, thick] node { } (n1);      
 \end{tikzpicture} 
&  
 $(n)$  \begin{tikzpicture}
 [scale=.15,auto=left, node distance=1.5cm, 
 ]
 \node[fill=white,style={circle,draw}] (n1) at (4,0) {\small{1}};
  \node[fill=black!30,style={circle,draw}] (n2) at (14,0) {\small{2}};
\path
        (n1)  [->]  edge[loop left=90,thick] node {} (n1)
         (n1) [->,thick]  edge[bend right=10] node { } (n2);
\path          
         (n2)  [->]  edge[loop right=90,thick] node {} (n2)       
        (n2) [->,thick]  edge[bend left=-10] node { } (n1);
 \end{tikzpicture} 
&
  $(n.1)$  \begin{tikzpicture}
 [scale=.15,auto=left, node distance=1.5cm, 
 ]
 \node[fill=white,style={circle,draw}] (n1) at (4,0) {\small{1}};
  \node[fill=black!30,style={circle,draw}] (n2) at (14,0) {\small{2}};
\path
        (n1)  [->]  edge[loop left=90,thick] node {} (n1)
         (n1) [->,thick]  edge[bend right=10] node { } (n2);
\path          
         (n2)  [->, dashed]  edge[loop right=90,thick] node {} (n2);
  \path      
        (n2) [->,thick]  edge[bend left=-10] node { } (n1);
 \end{tikzpicture} \\
\\
  $(n.2)$  \begin{tikzpicture}
 [scale=.15,auto=left, node distance=1.5cm, 
 ]
 \node[fill=white,style={circle,draw}] (n1) at (4,0) {\small{1}};
  \node[fill=black!30,style={circle,draw}] (n2) at (14,0) {\small{2}};
\path
        (n1)  [->]  edge[loop left=90,thick] node {} (n1);
\path 
         (n1) [->,dashed]  edge[bend right=10] node { } (n2);
\path          
         (n2)  [->]  edge[loop right=90,thick] node {} (n2);
 \path        
        (n2) [->,thick]  edge[bend left=-10] node { } (n1);
 \end{tikzpicture}  
&
  $(n.3)$  \begin{tikzpicture}
 [scale=.15,auto=left, node distance=1.5cm, 
 ]
 \node[fill=white,style={circle,draw}] (n1) at (4,0) {\small{1}};
  \node[fill=black!30,style={circle,draw}] (n2) at (14,0) {\small{2}};
\path
        (n1)  [->]  edge[loop left=90,thick] node {} (n1);
 \path 
         (n1) [->,dashed]  edge[bend right=10] node { } (n2);
\path          
         (n2)  [->, dashed]  edge[loop right=90,thick] node {} (n2);
 \path        
        (n2) [->,thick]  edge[bend left=-10] node { } (n1);
 \end{tikzpicture}  
&
  $(n.4)$  \begin{tikzpicture}
 [scale=.15,auto=left, node distance=1.5cm, 
 ]
 \node[fill=white,style={circle,draw}] (n1) at (4,0) {\small{1}};
  \node[fill=black!30,style={circle,draw}] (n2) at (14,0) {\small{2}};
\path
        (n1)  [->]  edge[loop left=90,thick] node {} (n1);
 \path 
         (n1) [->,dashed]  edge[bend right=10] node { } (n2);
\path          
         (n2)  [->,thick] edge[loop right=90,thick] node {} (n2);
 \path      
        (n2) [->,thick, dashed]  edge[bend left=-10] node { } (n1);
 \end{tikzpicture} 
 &
 $(o)$ \begin{tikzpicture}
 [scale=.15,auto=left, node distance=1.5cm, 
 ]
 \node[fill=white,style={circle,draw}] (n1) at (4,0) {\small{1}};
  \node[fill=black!30,style={circle,draw}] (n2) at (14,0) {\small{2}};
\path
        (n1)  [->]  edge[bend right=10,thick] node {} (n2)
         (n1) [->,thick]  edge[bend right=40, thick] node { } (n2);
 \path 
         (n2)  [->]  edge[bend left=-10,thick] node {} (n1)      
        (n2) [->,thick]  edge[bend left=-40, thick] node { } (n1);
 \end{tikzpicture} \\
\\ 
 $(o.1)$ \begin{tikzpicture}
 [scale=.15,auto=left, node distance=1.5cm, 
 ]
 \node[fill=white,style={circle,draw}] (n1) at (4,0) {\small{1}};
  \node[fill=black!30,style={circle,draw}] (n2) at (14,0) {\small{2}};
\path
        (n1)  [->]  edge[bend right=10,thick] node {} (n2);
  \path 
         (n1) [->,dashed]  edge[bend right=40, thick] node { } (n2);
 \path 
         (n2)  [->]  edge[bend left=-10,thick] node {} (n1)      
        (n2) [->,thick]  edge[bend left=-40, thick] node { } (n1);
 \end{tikzpicture} 
&
 $(o.3)$ \begin{tikzpicture}
 [scale=.15,auto=left, node distance=1.5cm, 
 ]
 \node[fill=white,style={circle,draw}] (n1) at (4,0) {\small{1}};
  \node[fill=black!30,style={circle,draw}] (n2) at (14,0) {\small{2}};
\path
        (n1)  [->,thick]  edge[bend right=10,thick] node {} (n2);
 \path 
         (n1) [->,dashed]  edge[bend right=40, thick] node { } (n2);
 \path 
         (n2)  [->,thick]  edge[bend left=-10,thick] node {} (n1);
\path      
        (n2) [->,dashed]  edge[bend left=-40, thick] node { } (n1);
 \end{tikzpicture} &&
\end{tabular}
\end{center}
}
\caption{$38$ connected 2-node UEI networks with input valence $\leq 2$. 
This set is partitioned into $4$ ODE-classes as in 
Table~\ref{table:UEI+ODE}. Representatives of the ODE-classes are in Figure~\ref{fig:ODE_2NCNUEIV2}.
}
\label{fig:2NCNUEIV2}
\end{figure}
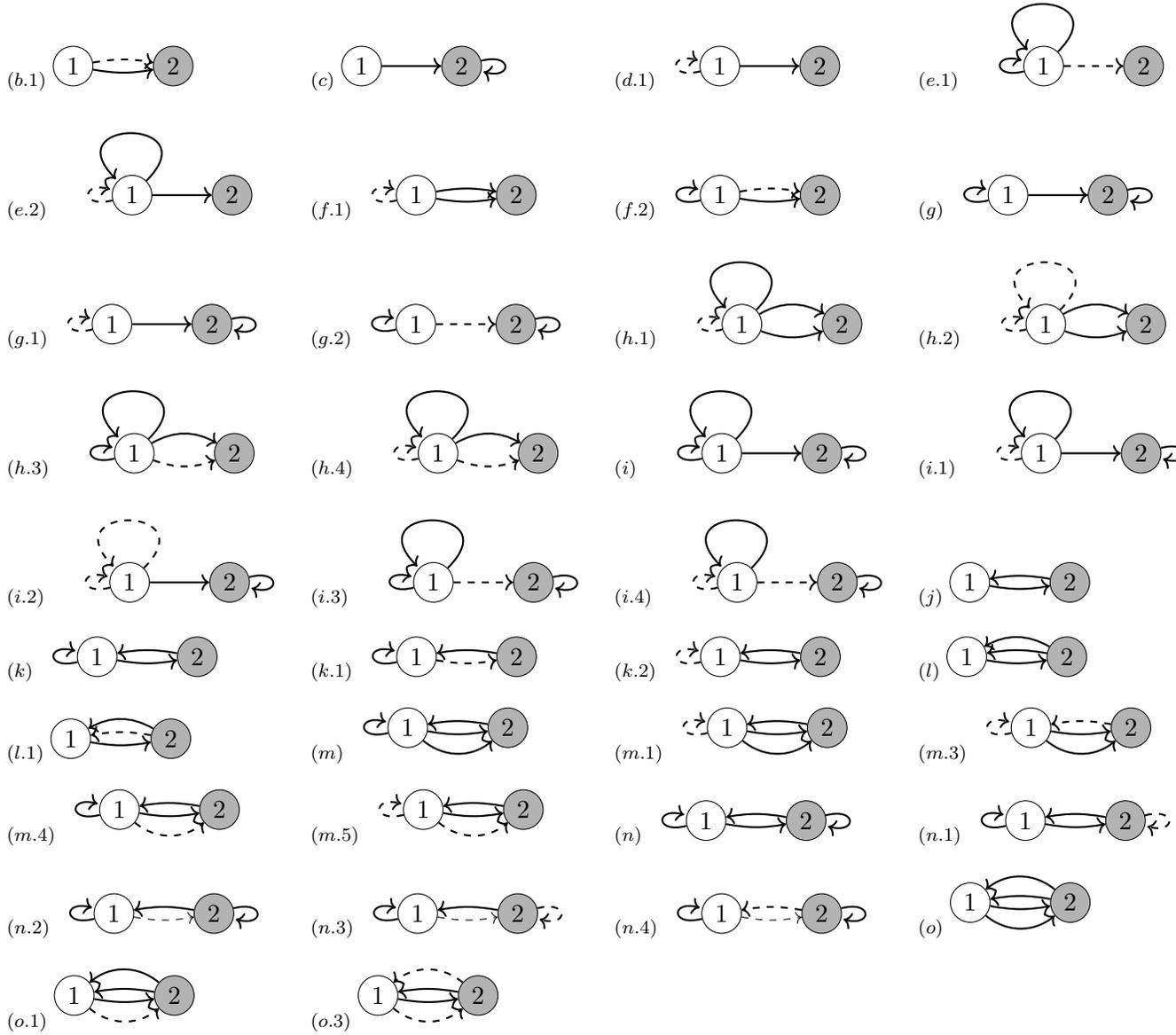

\begin{prop} 
There are $53$ connected $2$-node UEI networks with valence $\leq 2$: the $15$ REI networks in Figure~{\rm \ref{fig:2NCNREIV2}}, given by Proposition~\ref{prop:REI2},
and the $38$ networks in {\rm Figure~\ref{fig:2NCNUEIV2}}.  
Considering the 4 ODE-classes in Figure~\ref{fig:ODE_2NCNUEIV2},
the ODE-class {\rm NH1} contains $9$ of the UEI networks in Figure~{\rm \ref{fig:2NCNREIV2}} and the ODE-class {\rm NH2} contains the other $6$. 
The partition of $38$ UEI networks in {\rm Figure~\ref{fig:2NCNUEIV2}} into $4$ ODE-classes is stated in Table~{\rm \ref{table:UEI+ODE}}. 
\label{prop:UEI2}
\end{prop}

\begin{proof}
By Definition~\ref{def:EIN}, REI networks are 
UEI networks. Thus, we start by considering the $15$ connected $2$-node UEI networks with valence $\leq 2$ in Figure~\ref{fig:2NCNREIV2}, given by Proposition \ref{prop:REI2}. 
Next, from those $15$ networks we get the remaining UEI networks with valence $\leq 2$, up to renumbering of nodes and duality.
We maintain the assumption 
 that node $1$ is excitatory and node $2$ is inhibitory, but since we work with UEI networks we remove the 
node-type restriction on outputs. The result is Figure~\ref{fig:2NCNUEIV2}. 

Representatives of the minimal
ODE-classes of the 53 networks are in Figure~\ref{fig:ODE_2NCNUEIV2}. 
By Proposition~\ref{prop:REI2}, the set of $15$ networks in Figure~\ref{fig:2NCNREIV2}  is partitioned into the
ODE-classes NH1 and NH2.   By the results of \cite{DS05},
the set of $38$ networks in Figure~\ref{fig:2NCNUEIV2} is partitioned into 
the $4$ ODE-classes according to Table~\ref{table:UEI+ODE}. 
Table~\ref{table:ODE_2NCNUEIV2} states the corresponding admissible ODEs. 
\end{proof}

\begin{table}[!h]
\begin{tabular}{|ll|ll|}
\hline 
NH1 & $(b.1)-(i.4)$ 
& NH2 & $(k.1), (l.1), (m.3)-(m.5)$, \\
 &  & & $(n.2)-(n.3),(o.1)$
  \\
\hline
NH3
 & $(j)-(k), (k.2), $
&  
NH4
& $(l), (m)-(m.1)$\\
& $(n)-(n.1), (n.4), (o), (o.3)$ && \\
\hline
\end{tabular}
\vspace{.2cm}
\caption{Partition of the connected $2$-node UEI networks with valence $\leq 2$, listed in Figure~\ref{fig:2NCNUEIV2}, into the four 
network ODE-classes in Figure~\ref{fig:ODE_2NCNUEIV2}.}
\label{table:UEI+ODE}
\end{table}

\begin{cor} 
Only one  
ODE class of connected $2$-node UEI networks of valence $\leq 2$ has a minimal representative with two types of arrow, namely 
{\rm NH2}. This is also an REI network. 
\end{cor}

\begin{rems}  \label{rmk:UEI_classes}
(i) There are $2$ ODE-classes of connected $2$-node UEI networks of valence $\leq 2$ which coincide with the 
$2$ ODE-classes of connected $2$-node REI networks of valence $\leq 2$. 

(ii) There are $2$ ODE-classes  of connected $2$-node UEI networks of valence $\leq 2$ which are not REI. Moreover, they have representatives with only one 
arrow-type.

\hfill $\Diamond$
\end{rems}

\subsection{Connected 2-node CEI Networks}
\label{S:C2CEI}

A $2$-node connected CEI network is the union of two subnetworks as in Figure~\ref{fig:generalUEI}, but considering the two nodes to be of the same type.
The arrow multiplicities are nonnegative integers $\alpha_i, \beta_i, \gamma_i, \delta_i$, where $i=1,2$ and at least one of the $\beta_1,\beta_2, \gamma_1, \gamma_2$ is nonzero,  see~Figure~\ref{fig:generalCEI}.

\begin{figure}
\begin{tabular}{ll}
\begin{tikzpicture}
 [scale=.15,auto=left, node distance=2cm
 ]
 \node[fill=white,style={circle,draw}] (n1) at (4,0) {\small{1}};
  \node[fill=white,style={circle,draw}] (n2) at (20,0) {\small{2}};
  \path
   (7,1)  [->] edge[thick] node {$\beta_1$}  (17,1)
 (n1)  [line width=1pt,->]  edge[loop left=90,thick] node {$\alpha_1$} (n1);
 \path 
 (17,-1)  [->] edge[thick] node {$\beta_2$}  (7,-1)
 (n2)  [->]  edge[loop right=90,thick] node {$\alpha_2$} (n2); 
 \end{tikzpicture}  & 
 \begin{tikzpicture}
 [scale=.15,auto=left, node distance=2cm
 ]
 \node[fill=white,style={circle,draw}] (n1) at (4,0) {\small{1}};
  \node[fill=white,style={circle,draw}] (n2) at (20,0) {\small{2}};
  \path
   (7,1)  [dashed,->] edge[thick] node {$\gamma_1$}  (17,1)
 (n1)  [line width=1pt,->]  edge[loop left=90,thick] node {$\delta_1$} (n1);
 \path 
 (17,-1)  [dashed, ->] edge[thick] node {$\gamma_2$}  (7,-1)
 (n2)  [->]  edge[loop right=90,thick] node {$\delta_2$} (n2); 
  \end{tikzpicture} 
 \end{tabular}
 \caption{A $2$-node CEI network has two $2$-node subnetworks.} \label{fig:generalCEI}
\end{figure}
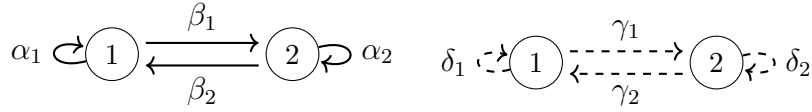

There is an analogous result to Proposition~\ref{lem:general_2_node_UEI} for UEI 
networks with the same proof.

\begin{prop} \label{lem:general_2_node_CEI} 
A $2$-node connected CEI network is the  network of Figure~{\rm \ref{fig:generalCEI}}, for some choice of nonnegative integer arrow multiplicities $\alpha_i, \beta_i, \gamma_i, \delta_i$, where $i=1,2$ and at least one of the $\beta_1,\beta_2, \gamma_1, \gamma_2$ is nonzero. All the statements of Proposition~{\rm \ref{lem:general_2_node_UEI}} hold, with the additional condition that  the two nodes have the same node-type.
When the multiplicities satisfy $\alpha_1 + \beta_2 = \alpha_2 + \beta_1$ and $\delta_1 + \gamma_2 = \gamma_1 + \delta_2$,  then the network is homogeneous and {\rm (\ref{eq:general2UEI})} holds with $f=g$.
\end{prop}

\subsection{Connected 2-node CEI Networks: ODE-classes}
\label{S:C2CEINODE}
For a 2-node CEI network there is no restriction on the tail node-type for $A^E$ arrows and $A^I$ arrows, and the two nodes have the same type. The adjacency matrices are therefore
$$
\begin{array}{ll} 
\mbox{Node-type $N^E=N^I$: } 
\id_2; &   \\ 
\ \\
\mbox{Arrow-type $A^E$: } 
A_3 = 
\left[
\begin{array}{cc}
\alpha_1 & \beta_2 \\
\beta_1 & \alpha_2 
\end{array}
\right];  & 
\mbox{Arrow-type $A^I$: } 
A_4= 
\left[
\begin{array}{cc}
\delta_1& \gamma_2 \\
\gamma_1 & \delta_2
\end{array}
\right];
\end{array}
$$
where at least one of the $\beta_i$ or $\gamma_i$ is nonzero to guarantee connectedness.

Following Definition 4.2 in \cite{AD07}, given a network $G$ and the corresponding ODE-class $[G]$, we 
 denote by $m[G]$ the minimal number of edges for the networks in  $[G]$.
Following Definition  5.10 in \cite{AD07}, given a matrix $M = [m_{ij}]_{1 \le i,j \le n} \in  M_{n×n}(\R)$, let
$l(M) =\sum_{i=1}^n \sum_{j=1}^n m_{ij}$.
\cite[Proposition 5.11]{AD07} implies:

\begin{prop} \label{prop:general_CEI2}
There is an infinity of ODE-classes of connected $2$-node CEI networks. 
Moreover, given a $2$-node CEI network $G$ with arrow adjacency matrices $A_3$ and $A_4$, let 
$$m = \dim \langle\id_2, A_3, A_4\rangle - 1\, .$$
Then a minimal EI network ODE-equivalent to $G$ has arrow adjacency matrices $M_1, M_2$
such that: 

{\rm (i)} $\{\id_2,M_1, M_2\}$ is a basis of the real vector space $\langle\id_2, A_3, A_4\rangle$.

{\rm (ii)} $\sum_{k=1}^m l(M_k) = m[G]$.
\end{prop}

\subsection{Connected 2-node CEI Networks with valence up to 2}
\label{S:C2CEINV2}

\begin{prop} \label{prop:CEI2}
The set of $2$-node connected CEI networks with node input valence up to two contains 
$53$ networks, which correspond to the $15$ UEI networks in Figure~{\rm \ref{fig:2NCNREIV2}} and the $38$ UEI networks in Figure~{\rm \ref{fig:2NCNUEIV2}}, given by Proposition~\ref{prop:UEI2}, 
 but assuming the nodes to be of the same type. 

Moreover, it is partitioned into 
$21$ ODE-classes. The $15$ CEI networks in Figure~{\rm \ref{fig:2NCNREIV2}} are PEI networks and are partitioned into
$9$ ODE-classes: $7$ classes
formed by inhomogeneous networks and $2$ by homogeneous networks, see Figure~{\rm \ref{fig:ODE_2NCNPEIV2}}.
The $38$ CEI networks in Figure~{\rm \ref{fig:2NCNUEIV2}} are partitioned into 
$15$ ODE-classes, $3$ of them PEI ODE-classes: 
 $11$
 classes formed by inhomogeneous networks and $4$ by homogeneous networks, see Figure~{\rm \ref{fig:ODE_2NCNCEIV2}}. 
The  corresponding admissible maps appear in Table~{\rm \ref{table:ODE_2NCNCEIV2}}.
\end{prop}

\begin{proof}
Trivially, the 2-node connected CEI networks with node input valence up to 2 can be obtained from the 2-node connected UEI networks with node input valence up to 2, 
see Figures~{\rm \ref{fig:2NCNREIV2}} and \ref{fig:2NCNUEIV2}, by considering the nodes to have the same type. 
The $15$ CEI networks in Figure~{\rm \ref{fig:2NCNREIV2}} are PEI networks.  By Proposition~\ref{prop:PEI2}, they are partitioned into
$9$ ODE-classes: $7$ classes formed by inhomogeneous networks and $2$ by homogeneous networks, see Figure~{\rm \ref{fig:ODE_2NCNPEIV2}}.

By the result in~\cite{DS05} on networks ODE-equivalence, 
the $38$ CEI networks in Figure~{\rm \ref{fig:2NCNUEIV2}}
are partioned into
$15$ ODE-classes of CEI networks, where 3 of them are PEI ODE-classes: 
one class containing  the networks $(b.1), (g), (g.2)$, one class with networks $(d.1), (e.2), (f.1), (f.2), (h.1), (h.2), (i.1), (i.3),$ $(e.1), (h.3)$, one class with networks $(g.1), (i.2)$, one class with networks $(h.4), (i), (i.4)$, one class with networks $(j), (n), (n.4), (o), (o.3)$, one class with networks $(k.1), (m.3), (m.4)$, one with networks $(l.1), (n.2), (o.1)$,
and one with $(k.2)$, $(n.1)$.
Each of the remaining networks, $(c)$, $(k)$, $(l)$, $(m)$, $(m.1)$, $(m.5)$, $(n.3)$, represents a different ODE-class. 
See Figure~\ref{fig:ODE_2NCNCEIV2} for minimal ODE-class representatives and Table~\ref{table:ODE_2NCNCEIV2} for the admissible maps.
\end{proof}

\begin{figure}[!ht]
\tiny{
\begin{center}
\begin{tabular}{llll} 
 NH1 \begin{tikzpicture}
 [scale=.15,auto=left, node distance=1.5cm, 
 ]
 \node[fill=white,style={circle,draw}] (n1) at (4,0) {\small{1}};
  \node[fill=white,style={circle,draw}] (n2) at (14,0) {\small{2}};
\path
        (n1) [->,solid, thick]  edge[thick] node { } (n2); 
 \end{tikzpicture}   
&
 NH2 \begin{tikzpicture}
 [scale=.15,auto=left, node distance=1.5cm, 
 ]
 \node[fill=white,style={circle,draw}] (n1) at (4,0) {\small{1}};
  \node[fill=white,style={circle,draw}] (n2) at (14,0) {\small{2}};
\path
        (n1) [->,solid, thick]  edge[thick] node { } (n2);
\path        
        (n2) [->,thick]  edge[loop right=90,thick] node { } (n2);
 \end{tikzpicture} 
 &
 NH3 \begin{tikzpicture}
 [scale=.15,auto=left, node distance=1.5cm, 
 ]
 \node[fill=white,style={circle,draw}] (n1) at (4,0) {\small{1}};
  \node[fill=white,style={circle,draw}] (n2) at (14,0) {\small{2}};
\path
        (n1) [->,thick]  edge[thick] node { } (n2);
 \path 
        (n1) [->,thick,dashed]  edge[loop left=90,thick] node { } (n1);
 \end{tikzpicture} 
& 
NH5 \begin{tikzpicture}
 [scale=.15,auto=left, node distance=1.5cm, 
 ]
 \node[fill=white,style={circle,draw}] (n1) at (4,0) {\small{1}};
  \node[fill=white,style={circle,draw}] (n2) at (14,0) {\small{2}};
\path
        (n1) [->,thick]  edge[thick] node { } (n2);
 \path 
         (n1) [->,dashed]  edge[loop left=90, thick] node { } (n1);
\path 
        (n2) [->,thick]  edge[loop right=90, thick] node { } (n2);
 \end{tikzpicture}\\
\\
 NH8 \begin{tikzpicture}
 [scale=.15,auto=left, node distance=1.5cm, 
 ]
 \node[fill=white,style={circle,draw}] (n1) at (4,0) {\small{1}};
  \node[fill=white,style={circle,draw}] (n2) at (14,0) {\small{2}};
\path
        (n1)  [->]  edge[loop left=90,thick] node {} (n1)
        (n1) [->,thick]  edge[bend right=10, thick] node { } (n2);
\path 
        (n2) [->,thick]  edge[bend left=-10] node { } (n1);
 \end{tikzpicture}
&
 NH9 \begin{tikzpicture}
 [scale=.15,auto=left, node distance=1.5cm, 
 ]
 \node[fill=white,style={circle,draw}] (n1) at (4,0) {\small{1}};
  \node[fill=white,style={circle,draw}] (n2) at (14,0) {\small{2}};
\path
        (n1)  [->]  edge[loop left=90,thick] node {} (n1);
 \path 
        (n1) [->,dashed]  edge[bend right=10, thick] node { } (n2);
\path 
        (n2) [->,thick]  edge[bend left=-10] node { } (n1);
 \end{tikzpicture} 
&
NH10 \begin{tikzpicture}
 [scale=.15,auto=left, node distance=1.5cm, 
 ]
 \node[fill=white,style={circle,draw}] (n1) at (4,0) {\small{1}};
  \node[fill=white,style={circle,draw}] (n2) at (14,0) {\small{2}};
\path
        (n1)  [->,dashed]  edge[loop left=90,thick] node {} (n1);
 \path 
        (n1) [->,thick]  edge[bend right=10, thick] node { } (n2);
\path 
        (n2) [->,thick]  edge[bend left=-10] node { } (n1);
 \end{tikzpicture}
&
NH11 \begin{tikzpicture}
 [scale=.15,auto=left, node distance=1.5cm, 
 ]
 \node[fill=white,style={circle,draw}] (n1) at (4,0) {\small{1}};
  \node[fill=white,style={circle,draw}] (n2) at (14,0) {\small{2}};
\path
         (n1) [->,thick]  edge[bend right=10, thick] node { } (n2);
\path 
        (n2) [->,thick]  edge[bend left =-10] node { } (n1)
        (n2) [->,thick]  edge[bend left=-30] node { } (n1);
 \end{tikzpicture} \\
\\
NH12 \begin{tikzpicture}
 [scale=.15,auto=left, node distance=1.5cm, 
 ]
 \node[fill=white,style={circle,draw}] (n1) at (4,0) {\small{1}};
  \node[fill=white,style={circle,draw}] (n2) at (14,0) {\small{2}};
\path
         (n1) [->,thick]  edge[bend right=10, thick] node { } (n2);
\path 
        (n2) [->,thick, dashed]  edge[bend left =-10] node { } (n1);
 \end{tikzpicture}  
 &
NH14   \begin{tikzpicture}
 [scale=.15,auto=left, node distance=1.5cm, 
 ]
 \node[fill=white,style={circle,draw}] (n1) at (4,0) {\small{1}};
  \node[fill=white,style={circle,draw}] (n2) at (14,0) {\small{2}};
\path
        (n1)  [->,dashed]  edge[loop left=90,thick] node {} (n1);
\path 
        (n1) [->,thick]  edge[bend right=10, thick] node { } (n2)
         (n1) [->,thick]  edge[bend right=40, thick] node { } (n2);
\path 
        (n2) [->,thick]  edge[bend left=-10, thick] node { } (n1);      
 \end{tikzpicture}
 &
 NH17  \begin{tikzpicture}
 [scale=.15,auto=left, node distance=1.5cm, 
 ]
 \node[fill=white,style={circle,draw}] (n1) at (4,0) {\small{1}};
  \node[fill=white,style={circle,draw}] (n2) at (14,0) {\small{2}};
\path
        (n1)  [->]  edge[loop left=90,thick] node {} (n1);
 \path 
         (n1) [->,dashed]  edge[bend right=10] node { } (n2);
\path          
         (n2)  [->, dashed]  edge[loop right=90,thick] node {} (n2);
 \path        
        (n2) [->,thick]  edge[bend left=-10] node { } (n1);
 \end{tikzpicture}  
 & 
 \\
\\
 H1 
  \begin{tikzpicture}
 [scale=.15,auto=left, node distance=1.5cm, 
 ]
 \node[fill=white,style={circle,draw}] (n1) at (4,0) {\small{1}};
  \node[fill=white,style={circle,draw}] (n2) at (14,0) {\small{2}};
 \path 
         (n1) [->,thick]  edge[loop left=90, thick] node { } (n1);
\path
        (n1) [->,solid, thick]  edge[thick] node { } (n2);
 \end{tikzpicture}  
&
H2 \begin{tikzpicture}
 [scale=.15,auto=left, node distance=1.5cm, 
 ]
 \node[fill=white,style={circle,draw}] (n1) at (4,0) {\small{1}};
  \node[fill=white,style={circle,draw}] (n2) at (14,0) {\small{2}};
\path
        (n1) [->,thick]  edge[bend right=10] node { } (n2);
 \path 
        (n2) [->,thick]  edge[bend left=-10] node { } (n1);
 \end{tikzpicture}
&
 H3   \begin{tikzpicture}
 [scale=.15,auto=left, node distance=1.5cm, 
 ]
 \node[fill=white,style={circle,draw}] (n1) at (4,0) {\small{1}};
  \node[fill=white,style={circle,draw}] (n2) at (14,0) {\small{2}};
\path
        (n1)  [->]  edge[loop left=90,thick] node {} (n1)
        (n1) [->,thick]  edge[bend right=10, thick] node { } (n2)
         (n1) [->,thick]  edge[bend right=40, thick] node { } (n2);
\path 
        (n2) [->,thick]  edge[bend left=-10, thick] node { } (n1);      
 \end{tikzpicture}
&
 H4   \begin{tikzpicture}
 [scale=.15,auto=left, node distance=1.5cm, 
 ]
 \node[fill=white,style={circle,draw}] (n1) at (4,0) {\small{1}};
  \node[fill=white,style={circle,draw}] (n2) at (14,0) {\small{2}};
\path
        (n1)  [->,dashed]  edge[loop left=90,thick] node {} (n1);
 \path 
        (n1) [->,thick]  edge[bend right=10, thick] node { } (n2);
 \path 
         (n1) [->,dashed]  edge[bend right=40, thick] node { } (n2);
\path 
        (n2) [->,thick]  edge[bend left=-10, thick] node { } (n1);      
 \end{tikzpicture} 
\end{tabular}
\end{center}
\caption{
Minimal representatives of the $15$
 ODE-classes of the $38$ 
 $2$-node connected CEI networks with input valence up to two in Figure~{\rm \ref{fig:2NCNUEIV2}} (assuming the nodes to be of the same type): 
 $11$ inhomogeneous and 4 homogeneous. 
See Table~\ref{table:ODE_2NCNCEIV2} for the corresponding admissible maps. Representatives NH1, H1, and NH12 are PEI networks.}
\label{fig:ODE_2NCNCEIV2}
}
\end{figure}
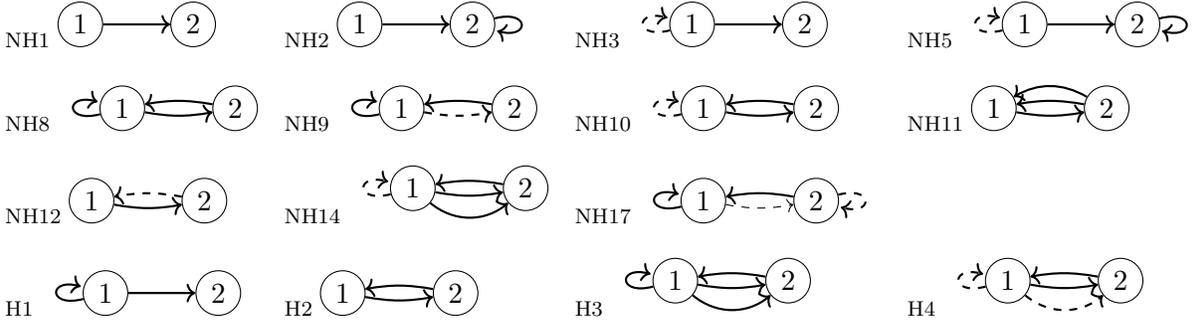

\begin{table}
\begin{center}
{\tiny 
\begin{tabular}{|l|l|l|}
\hline 
 & &  \\
NH1 \quad $
\begin{array}{l}
\dot{x}_1 = f(x_1) \\
\dot{x}_2 = g(x_2; x^+_1)
\end{array}
$ 
& 
NH2 \quad$
\begin{array}{l}
\dot{x}_1 = f(x_1) \\
\dot{x}_2 = g(x_2; \overline{x^+_1,x^+_2})
\end{array}
$  
& 
NH3 \quad$
\begin{array}{l}
\dot{x}_1 = f(x_1; x^-_1) \\
\dot{x}_2 = g(x_2; x^+_1)
\end{array}
$  
\\
 && \\
\hline 
 && \\
NH5 \quad$
\begin{array}{l}
\dot{x}_1 = f(x_1;x^-_1) \\
\dot{x}_2 = g(x_2; \overline{x^+_1,x^+_2})
\end{array}
$  
 &
NH8\quad
$
\begin{array}{l}
\dot{x}_1 = f(x_1; \overline{x^+_1,x^+_2}) \\
\dot{x}_2 = g(x_2; x^+_1)
\end{array}
$ 
& 
NH9\quad
$ 
\begin{array}{l}
\dot{x}_1 = f(x_1; \overline{x^+_1,x^+_2}) \\
\dot{x}_2 = g(x_2; x^-_1)
\end{array}
$ 
\\
 && \\
\hline 
 && \\
NH10\quad
$
\begin{array}{l}
\dot{x}_1 = f(x_1; x^-_1,x^+_2) \\
\dot{x}_2 = g(x_2; x^+_1)
\end{array}
$ 
& 
NH11\quad
$ 
\begin{array}{l}
\dot{x}_1 = f(x_1; \overline{x^+_2,x^+_2}) \\
\dot{x}_2 = g(x_2; x^+_1)
\end{array}
$ &
NH12\quad
$
\begin{array}{l}
\dot{x}_1 = f(x_1; x^-_2) \\
\dot{x}_2 = g(x_2; x^+_1)
\end{array}
$ 
\\
&& \\
\hline
&& \\
NH14\quad
$ 
\begin{array}{l}
\dot{x}_1 = f(x_1; x^-_1,x^+_2) \\
\dot{x}_2 = g(x_2; x^+_1, x^+_1)
\end{array}
$ 
& 
NH17\quad
$ 
\begin{array}{l}
\dot{x}_1 = f(x_1; \overline{x^+_1,x^+_2}) \\
\dot{x}_2 = g(x_2; \overline{x^-_1,x^-_2})
\end{array}
$
&
H1 \quad $
\begin{array}{l}
\dot{x}_1 = f(x_1, x^+_1) \\
\dot{x}_2 = f(x_2; x^+_1)
\end{array}
$ \\
&& \\
\hline
&& \\
 H2 \quad$
\begin{array}{l}
\dot{x}_1 = f(x_1,x^+_2) \\
\dot{x}_2 = f(x_2; x^+_1)
\end{array}
$  
& 
H3 \quad$
\begin{array}{l}
\dot{x}_1 = f(x_1;  \overline{x^+_1,x^+_2}) \\
\dot{x}_2 = f(x_2;  \overline{x^+_1,x^+_1})
\end{array}
$  
& 
H4 \quad$
\begin{array}{l}
\dot{x}_1 = f(x_1; x^-_1,x^+_2) \\
\dot{x}_2 = f(x_2; x^+_1,x^-_1)
\end{array}
$  
\\
 && \\
\hline
\end{tabular}
}
\vspace{.2cm}
\caption{
Admissible maps for the networks in Figure~\ref{fig:ODE_2NCNCEIV2}.}
\label{table:ODE_2NCNCEIV2}
\end{center}
\end{table}

\section{Conclusions}

This work is motivated by the importance of biological networks in science. Commonly,  in these networks, a distinction is made  on the type of connections (excitatory and inhibitory)   and on the type of nodes  (activator and repressor).  We make a general study of 2-node excitatory inhibitory networks (EI) with two types of connections and some more conditions.  More precisely, we formalize the structure of EI networks as a preparation of a systematic analysis of dynamics and bifurcations in such networks. Moreover, our results  are extended to the 3-node case in \cite{ADS3node}. We consider the network formalism where nodes and connections are partitioned into several types and where the dynamics of the networks respects these and the network topology.  In this paper we classify four classes of 2-node EI networks -- REI, UEI, PEI, CEI -- all with two types of connections (excitatory and inhibitory):  for the REI network class, a node cannot output both types of connections whereas for the UEI class a node can output both; moreover, for both REI and UEI networks, there are two node-types (activators and repressors); when all nodes are assumed to be of the same type, then we have the PEI and CEI network classes: if a node cannot output both types of connections the network is PEI, otherwise,  the network is CEI. The number of networks of every such network class is not finite. Trivially, restricting the network valence, then each of the four classes has a finite number of networks. 

Remarkably, considering the classification up to ODE-equivalence, we obtain that there are only 2 ODE-classes of REI networks while for 
the other three network classes, there are infinite number of ODE-classes.  

Restricting to networks of valency $\leq2$ then the four network classes are formed by  finite number of networks.  We obtain that the classification 
of CEI (resp. PEI) networks  is derived  from the classification of the UEI (resp. REI) networks by assuming the two nodes are of the same type, however,  for CEI (resp. PEI) networks there are 21 (resp. 9) ODE-classes and for the UEI (resp. REI) networks there are only 4 (resp. 2) ODE-classes. 

\vspace{5mm}

\noindent {\bf Acknowledgments} \\
MA and AD were partially supported by CMUP, member of LASI, which is financed by national funds through FCT -- Funda\c c\~ao para a Ci\^encia e a Tecnologia, I.P., under the projects with reference UIDB/00144/2020 and UIDP/00144/2020.

\end{document}